\documentclass[a4paper]{amsart}

\usepackage[utf8]{inputenc}
\usepackage{amsmath}
\usepackage{amssymb}
\usepackage{amsthm}
\usepackage{tikz}
\usepackage{graphicx}
\usepackage{float}
\usepackage{svg}
\usepackage{verbatim}
\usepackage{url}
\usepackage{mathtools}
\usepackage{microtype}
\usepackage{thm-restate}
\usepackage{enumitem}
\usepackage[hidelinks]{hyperref}
\usepackage[nameinlink]{cleveref}
\usepackage{caption}
\let\cref\Cref
\Crefname{enumi}{}{}

\theoremstyle{plain}
\newtheorem{theorem}{Theorem}[section]
\theoremstyle{definition}
\newtheorem{definition}[theorem]{Definition}
\theoremstyle{plain}
\newtheorem{lemma}[theorem]{Lemma}

\newtheorem{question}[theorem]{Question}
\theoremstyle{definition}
\newtheorem{example}[theorem]{Example}

\newenvironment{Example}{\begin{example}\rm}{\end{example}}
\theoremstyle{plain}
\newtheorem{proposition}[theorem]{Proposition}

\theoremstyle{remark}
\newtheorem{remark}[theorem]{Remark}

\newcommand{\qua}{\hskip 0.4em \ignorespaces}
\def\arxiv#1{\relax\ifhmode\unskip\qua\fi
\href{http://arxiv.org/abs/#1}%
{\tt arXiv:\penalty -100\unskip#1}}

\def\MR#1{\relax\ifhmode\unskip\qua\fi
\href{https://mathscinet.ams.org/mathscinet-getitem?mr=#1}{\tt MR#1}}
\def\ZB#1{\relax\ifhmode\unskip\qua\fi
\href{https://zbmath.org/?q=an:#1}{\tt Zbl\:#1}}
\def\xox#1{\csname xx#1\endcsname}

\renewenvironment{thebibliography}[1]{
  \begin{oldthebibliography}{#1}\small
    \setlength{\itemsep}{.5ex}
    \setlength{\parskip}{0em}
}
{
  \end{oldthebibliography}
}

\let\stdthebibliography\thebibliography
\let\stdendthebibliography\endthebibliography
\renewenvironment*{thebibliography}[1]{%
  \stdthebibliography{MM20.}}
  {\stdendthebibliography}

\newcommand{\R}{\mathbb{R}}
\DeclareMathOperator{\cl}{cl}

\title{Homogeneous braids are visually prime}

\author{Peter Feller}
\author{Lukas Lewark}
\author{Miguel Orbegozo Rodriguez}
\address{Université de Neuchâtel, %Institut de mathématiques,
Rue Emile-Argand 11, 2000 Neuchâtel, Switzerland}
\email{peter.feller@unine.ch}
\urladdr{\url{https://www.unine.ch/math/en/pfeller/}}
\address{ETH Z\"urich, R\"amistrasse 101, 8092 Z\"urich, Switzerland}
\email{lukas.lewark@math.ethz.ch}
\urladdr{\url{https://people.math.ethz.ch/~llewark/}}
\address{Université de Neuchâtel, %Institut de mathématiques,
Rue Emile-Argand 11, 2000 Neuchâtel, Switzerland}
\email{miguel.orbegozo@unine.ch}
\urladdr{\url{https://sites.google.com/view/miguel-orbegozo-rodriguez}}
\begin{document}
%\subjclass[2020]{57K10,57K20,57K30}

\begin{abstract}
We show that closures of homogeneous braids are visually prime, addressing a question of Cromwell.
The key technical tool for the proof is the following criterion concerning primeness of open books, which we consider to be of independent interest. For open books of $3$--manifolds the property of having no fixed essential arcs is preserved under essential Murasugi sums with a strictly right-veering open book, if the plumbing region of the original open book veers to the left. %
We also provide examples of open books in $S^3$ demonstrating that %
primeness is not necessarily preserved under essential Murasugi sum, in fact not even under stabilizations a.k.a.\ Hopf plumbings. Furthermore, we find that trefoil plumbings need not preserve primeness. In contrast, we establish that figure-eight knot plumbings do preserve primeness.
\end{abstract}
\maketitle
\section{Introduction}

Knot theorists have a somewhat ambivalent relationship with the most common way of representing knots, i.e.~by knot diagrams. The latter are the result of projecting a knot---a circle embedded in~$\R^3$---to a generic $2$--dimensional plane retaining crossing information. Indeed, many simple 3--dimensional properties, e.g.~primeness, are often hard to discern from a diagram. However, for particular classes of knot diagrams, one can read off a given property directly from the diagram.
Concerning primeness, the two first key results are due to Menasco~\cite{Menasco_84} and Cromwell~\cite{Cromwell_93}, respectively: they show that primeness of links---embeddings of a collection of circles in~$\R^3$---represented by alternating and braid positive diagrams can be read off from such diagrams. More precisely, if the link is not prime, i.e.~if there exists a decomposition sphere, then there exists a decomposition sphere that arises from a decomposition circle in the diagram. These results are condensed into the slogan ``alternating links and positive braids are visually prime''. Later, Ozawa~\cite{Ozawa_02} extended Cromwell's result to positive diagrams.
Cromwell more generally conjectured a link diagram to be visually prime if its Seifert genus is realized by a Seifert surface coming from Seifert's algorithm. More precisely, given a link $L$ with a diagram $D$
for which Seifert's algorithm yields a minimal Seifert surface, the following ought to hold:
if $L$ is not prime, then $D$
features a decomposition circle that gives rise to a decomposition sphere~\cite%
{Cromwell_93}. Cromwell points out that for diagrams arising from braids, Seifert's algorithm yields a minimal Seifert surface if and only if the braid is homogeneous. An $n$--stranded braid word~$\beta$---an expression in Artin's standard generators $\sigma_1,\dots,\sigma_{n-1}$ and their inverses~\cite{Artin_25}---is \emph{homogeneous} if for each $i$ only one of $\sigma_i$  and $\sigma_i^{-1}$ occurs. We resolve Cromwell's conjecture in the positive for diagrams that arise from braids\footnote{We note that Stoimenow found a counterexample to the full conjecture \cite[Figure 15]{Counterexample_Stoimenow}.}, that is, we prove the following.

\begin{theorem}[Homogeneous braids are visually prime]\label{thm:main}

Let $L$ be the closure of a braid given by a homogeneous braid word and let $D$ be the link diagram associated with that braid word. If $L$ is not prime, then there exists a decomposition circle for $D$ that gives rise to a decomposition sphere for~$L$.
\end{theorem}

Our proof, much like Cromwell's proof for positive braids, uses the fiber structure. However, rather than drawing upon local pictures and fundamental group calculations, we prove an open book result (see next subsection and \Cref{sec:murasugi_sum}) and require from the knot theory side only a description of the fiber surface of a non-split homogeneous braid as an iterative Murasugi sum (see \Cref{sec:sketch} for a sketch, and see \Cref{sec:braids} for a detailed account).

\subsection*{A primeness criterion for open books}
Let $(\Sigma,\phi)$ be an open book---a pair consisting of an oriented compact surface $\Sigma$ with non-empty boundary, called the \emph{page}, and an orientation preserving self-diffeomorphism~$\phi$, called the \emph{monodromy}. 
For the rest of the introduction and~\Cref{sec:sketch}, we assume that the oriented $3$--manifold $M_\phi$ associated with the open book is irreducible (e.g.~$M_\phi\cong S^3$) for sake of exposition.
Here, the \emph{associated $3$--manifold} $M_\phi$ is the result of building the mapping torus of $\phi$ and collapsing its boundary components to a union of circles, called the \emph{binding}.
The binding, understood as a link in~$M_\phi$, is \emph{prime} if the open book has no fixed essential arcs, where an arc $a$---a properly embedded closed interval---in $\Sigma$ is said to be \emph{essential} if it cannot be isotoped relative boundary into~$\partial\Sigma$,
and is said to be \emph{fixed} if $\phi(a)$ and $a$ are isotopic relative boundary.

The so inclined reader might hope that primeness is preserved under stabilization as follows.
Assume $(\Sigma,\phi)$ has no fixed essential arcs and let $(\widetilde{\Sigma},\widetilde{\phi})$ be a stabilization (also known as Hopf plumbing) along some essential arc.  Then, surely, $(\widetilde{\Sigma},\widetilde{\phi})$ also has no fixed essential arc; in other words its binding is again a prime link in $M_{\widetilde{\phi}}\cong M_\phi$?!
This surmise is e.g.~true if $\phi$ is positive and the stabilization is positive, which was a key input in Ito's reproof of Cromwell's result for positive braids~\cite[Appendix~A]{Ito_22}; see also \Cref{rmk:connectiontoItosproof}.
However, it is \emph{false} in general, as the following example shows.%
\begin{Example}\label{ex:counterexample}
In this example we consider open books with associated $3$--manifold diffeomorphic to~$S^3$, which we describe by their bindings: fibered links in~$S^3$. Consider the prime fibered knot in~$S^3$
arising as closure of the homogeneous 3--stranded braid $\sigma_1^{-2} \sigma_2 \sigma_1^{-1} \sigma_2^2$, which is $6_3$ in the knot tables; see top left of \Cref{fig:6_3}.
 \begin{figure}[b]
    \centering
    \includegraphics[width=0.9\linewidth]{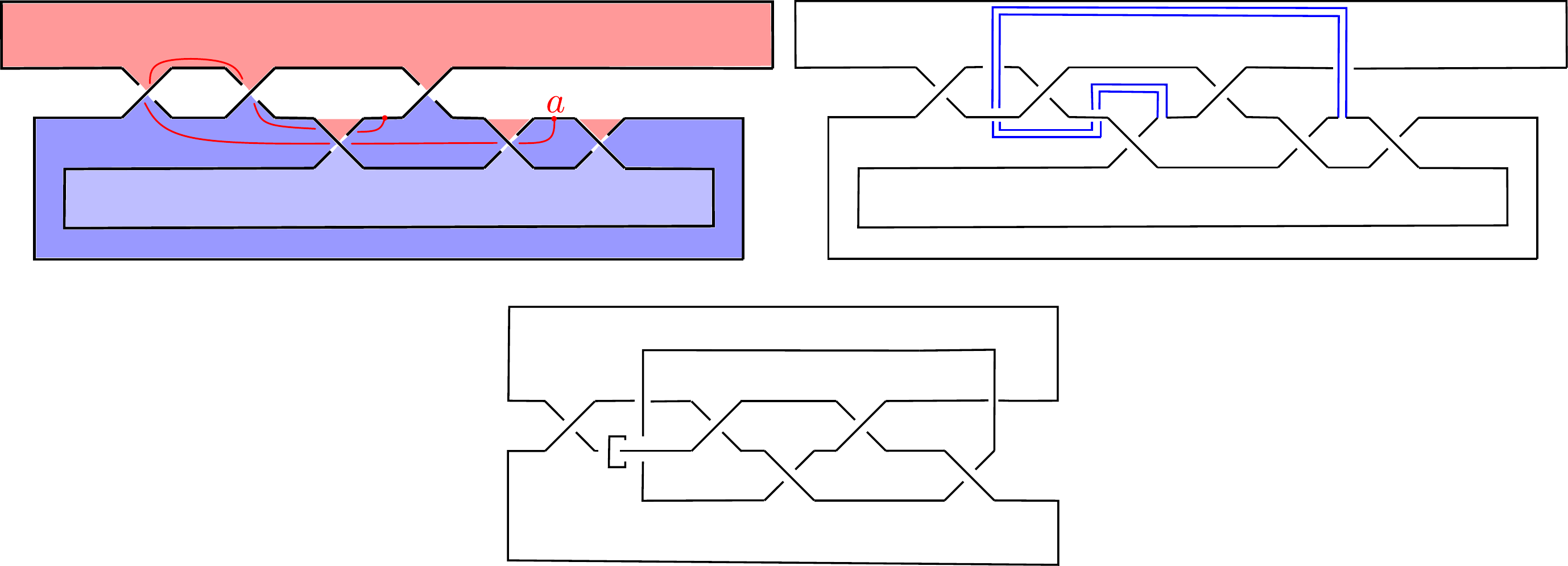}
    \caption{Top left: the knot~$6_3$ and the stabilizing arc $a$ (red).
    \newline
    Top right: the result of the stabilization (band being plumbed in blue).\newline
    Bottom: the result, after an isotopy, is seen to be a non-prime link: the connected sum of the prime fibered knot $8_{20}$ and a Hopf link.}
    \label{fig:6_3}
    \end{figure}
Stabilization along the essential arc $a$ depicted on the top left of \Cref{fig:6_3} results in a two component fibered link on the top right, which is not prime: it is the connected sum of a positive Hopf link and the prime fibered knot~$8_{20}$.

\end{Example}

The main work in our proof to establish \Cref{thm:main} lies in providing a criterion that assures that indeed
$(\widetilde{\Sigma},\widetilde{\phi})$ has no fixed essential arc. In fact, the criterion more generally works for Murasugi sums, a generalization of stabilization (which we recall below), and in the proof of \Cref{thm:main} we exploit the natural Murasugi sum structure of the closure of homogeneous braids.

Fixing open books $(\Sigma_1, \phi_1)$ and~$(\Sigma_2, \phi_2)$, and identifying some even-sided polygon $P$ with subsets of $\Sigma_1$ and~$\Sigma_2$, known as the \emph{summing region}, the \emph{Murasugi sum} of $(\Sigma_1, \phi_1)$ and $(\Sigma_2, \phi_2)$ along $P$ is the following open book. The page $\Sigma : = \Sigma_1 \#_{P} \Sigma_2$ is the result of gluing $\Sigma_2$ to $\Sigma_1$ along~$P$, and we take $\phi\coloneqq \phi_2\circ\phi_1\colon \Sigma\to \Sigma$ to be the monodromy, where $\phi_i\colon\Sigma\to \Sigma$ are the extensions of  $\phi_2\colon \Sigma_2\to\Sigma_2$ and $\phi_1\colon \Sigma_1\to\Sigma_1$ to $\Sigma$ by the identity. See \Cref{def:murasugissum} for a more precise definition.
The summing region is called \emph{essential} if the sides of $P$ are essential arcs in~$\Sigma$.
Murasugi sums generalize the notion of (positive) stabilization of an open book $(\Sigma_1,\phi_1)$ along an arc $a$ in~$\Sigma_1$.
To see this, take
$P$ to be a square and
$(\Sigma_2,\phi_2)$ to be an annulus with a (positive) Dehn twist along the core;
then, glue $\Sigma_1$ and $\Sigma_2$ by identifying $P$ with a product neighborhood of $a$ in~$\Sigma_1$,
and with a product neighborhood of the cocore of the annulus in~$\Sigma_2$.

\begin{restatable}%
{crit}{criterion}
\label{prop:nofixedarccriterion}
Let $(\Sigma_1, \phi_1)$ and $(\Sigma_2, \phi_2)$ be open books and let $P$ be an essential summing region.
 If all fixed essential arcs of the open book $(\Sigma_1, \phi_1)$ intersect~$P$, $(\Sigma_2, \phi_2)$ is strictly right-veering, and every arc in $P\subset \Sigma_1$ veers to the left or is fixed by~$\phi_1$, then the Murasugi sum of $(\Sigma_1, \phi_1)$ and $(\Sigma_2, \phi_2)$ along $P$ has no fixed essential arc.
\end{restatable}

With this criterion, the proof of \Cref{thm:main} reduces to showing that a non-split homogeneous braid closure with a diagram that features no decomposition circles arises as iterative Murasugi sums that satisfy the criterion. Said iterative Murasugi sum is rather natural from the braiding structure and is given by grouping together maximal positive and negative non-split subbraids. We provide a primer of this primeness argument in \Cref{sec:sketch}; the full argument is given in \Cref{sec:braids}.

\subsection*{Primeness of tree-like Murasugi sums}

As an illustration of \Cref{prop:nofixedarccriterion} we obtain that so-called fibered arborescent links---links arising from iterative plumbing of Hopf bands guided by a plane tree;
see~\cite[Ch.~12]{bs},~\cite{MR2165205}---are all prime. We see this as an illustration rather than a new result, and note that primeness, in fact, hyperbolicity is known for all but three exceptional families of arborescent knots; compare with Futer and Gu\'{e}ritaud's work~\cite{Futer_2008}.
Given that \Cref{prop:nofixedarccriterion} does not just work for~$S^3$, but is stated in terms of open books with any associated $3$--manifold, we take the opportunity to introduce a notion of a \emph{tree of open books}, essentially encoding the Murasugi summing of strictly right-veering and strictly left-veering open books along disjoint essential summing regions (see \Cref{sec:trees}). We then show, using the criterion, that the resulting open book has no fixed essential arcs, and is thus, in particular, prime (see \Cref{thm:prime_trees}). This simultaneously generalizes the key idea in the proofs of \Cref{thm:main} and of primeness of fibered arborescent links.

\subsection*{Trefoil plumbings do not preserve primeness, while figure-eight knot plumbings do.}

We discussed that the property of an open book having no fixed essential arcs is not necessarily preserved by essential Hopf plumbing; see \Cref{ex:counterexample}. It turns out that the same is true for essential trefoil plumbings, which we understand to mean the following. 
Starting with an open book~$(\Sigma,\phi)$, plumb a positive Hopf band along an essential arc $a$ in~$\Sigma$, and to the resulting open book $(\widetilde{\Sigma},\widetilde{\phi})$ plumb a second positive Hopf band along an essential arc that is disjoint from~$\Sigma\subset\widetilde{\Sigma}$.
As the second arc is unique up to isotopy, the resulting open book only depends on the choice of~$a$, and we call it the \emph{positive trefoil plumbing} of $({\Sigma},{\phi})$ along~$a$.
\begin{proposition}\label{prop:nonprimnessoftrefoilplumbing} There exists an open book $(\Sigma,\phi)$ with connected~$\partial\Sigma$, without fixed essential arcs, and with an essential arc $a\subset \Sigma$ such that positive trefoil plumbing of $(\Sigma,\phi)$ along $a$ has a fixed essential separating arc.
\end{proposition}
In other words,
there exists a prime fibered knot in a $3$--manifold $M_\phi$ that features an essential trefoil plumbing for which the resulting knot in $M_\phi$ is not prime.

Similarly, given an open book $(\Sigma,\phi)$ and an arc $a$ in~$\Sigma$, we call the same construction as above, but replacing the second positive Hopf band by a negative Hopf band, the \emph{figure-eight plumbing} of $({\Sigma},{\phi})$ along~$a$.
In contrast to trefoil plumbings, figure-eight plumbings preserve the property of having no fixed essential arcs, as a rather immediate application of \Cref{prop:nofixedarccriterion} shows.

\begin{proposition}
\label{prop:primnessoffigure8plumbing} Let $(\Sigma,\phi)$ be an open book and let $a$ be an essential arc in~$\Sigma$. If all fixed essential arcs of $(\Sigma,\phi)$ intersect $a$ (e.g.~if~$(\Sigma,\phi)$ has no fixed essential arcs), then the figure-eight plumbing of $(\Sigma,\phi)$ along $a$ has no fixed essential arc, and, in particular, is prime.
\end{proposition}

The authors were surprised by this dichotomy between trefoil plumbings and figure-eight plumbings. However, we do not know whether an open book as in \Cref{prop:nonprimnessoftrefoilplumbing} can be found such that the associated $3$--manifold is~$S^3$.

\begin{question}
Does there exist a prime fibered knot $K$ in $S^3$ to which one can plumb a trefoil along an essential arc in its fiber surface such that the resulting knot in $S^3$ is not prime?
\end{question}

\subsection*{Structure of the paper}
In~\Cref{sec:sketch}, we sketch a proof of~\Cref{thm:main} using \Cref{prop:nofixedarccriterion}.
In \Cref{sec:murasugi_sum}, we give detailed definitions concerning open books, Murasugi sums, veeringness, and prove \cref{prop:nofixedarccriterion}.
In \Cref{sec:trees}, we describe an iterative construction of Murasugi summing guided by a tree of open books, and we establish \Cref{thm:prime_trees}, which states that the resulting open book is prime, whenever all involved open books are strictly veering and the summing regions are essential. As an application, we discuss primeness of arborescent fibered links and prove \Cref{prop:primnessoffigure8plumbing}.
In \Cref{sec:braids}, we give an introduction to homogeneous braids and decomposition circles. Then we prove several lemmas on the Murasugi sum structure of homogeneous braids, which we use to give a detailed proof of \Cref{thm:main} by inductively showing that prime diagrams of homogeneous braids are obtained by repeated Murasugi sums satisfying the criterion. We do so using the language of trees of open books developed in \Cref{sec:trees}.
Finally, in \Cref{sec:trefoilplumbing}, we prove \Cref{prop:nonprimnessoftrefoilplumbing}.

\subsection*{Acknowledgments}
We would like to thank Mark Pencovitch for pointing out Stoimenow's example. PF and MOR gratefully acknowledge financial support by the SNSF Grant 181199.

\section{Outline of the proof of visual primeness of homogeneous braids}\label{sec:sketch}

We provide a sketch of the proof of \Cref{thm:main}, relying on what we hope the reader will find instructive examples, namely the $5$--stranded homogeneous braids  \[\beta_{\mathrm{prime}}
=\sigma_3\sigma_4^{-1}\sigma_1^{-2}\sigma_3^2\sigma_2^{2}\sigma_4^{-1}\sigma_1^{-1}\sigma_3^{2}\sigma_2\sigma_4^{-1}
\quad\text{and}\quad
\beta_{\mathrm{comp}}=\sigma_3^{5}\sigma_1^{-3}\sigma_2^{3}\sigma_4^{-3},\]
and invoking \Cref{prop:nofixedarccriterion}.
In this sketch we readily switch between open books $(\Sigma,\phi)$ for which the associated $3$--manifolds are diffeomorphic to $S^3$ and their associated fiber surfaces in~$S^3$.

\subsection*{Step 1: Diagram and fiber surface associated with a homogeneous braid}
Let $D(\beta)$ be the diagram arising from a homogeneous braid $\beta$ on $n$ strands; see left-hand side of \Cref{fig:braidanddiagr} for an illustration for $D(\beta_{\mathrm{prime}})$ and~$D(\beta_{\mathrm{comp}})$. 
\begin{figure}[ht]
    \centering
    \includegraphics[width=\linewidth]{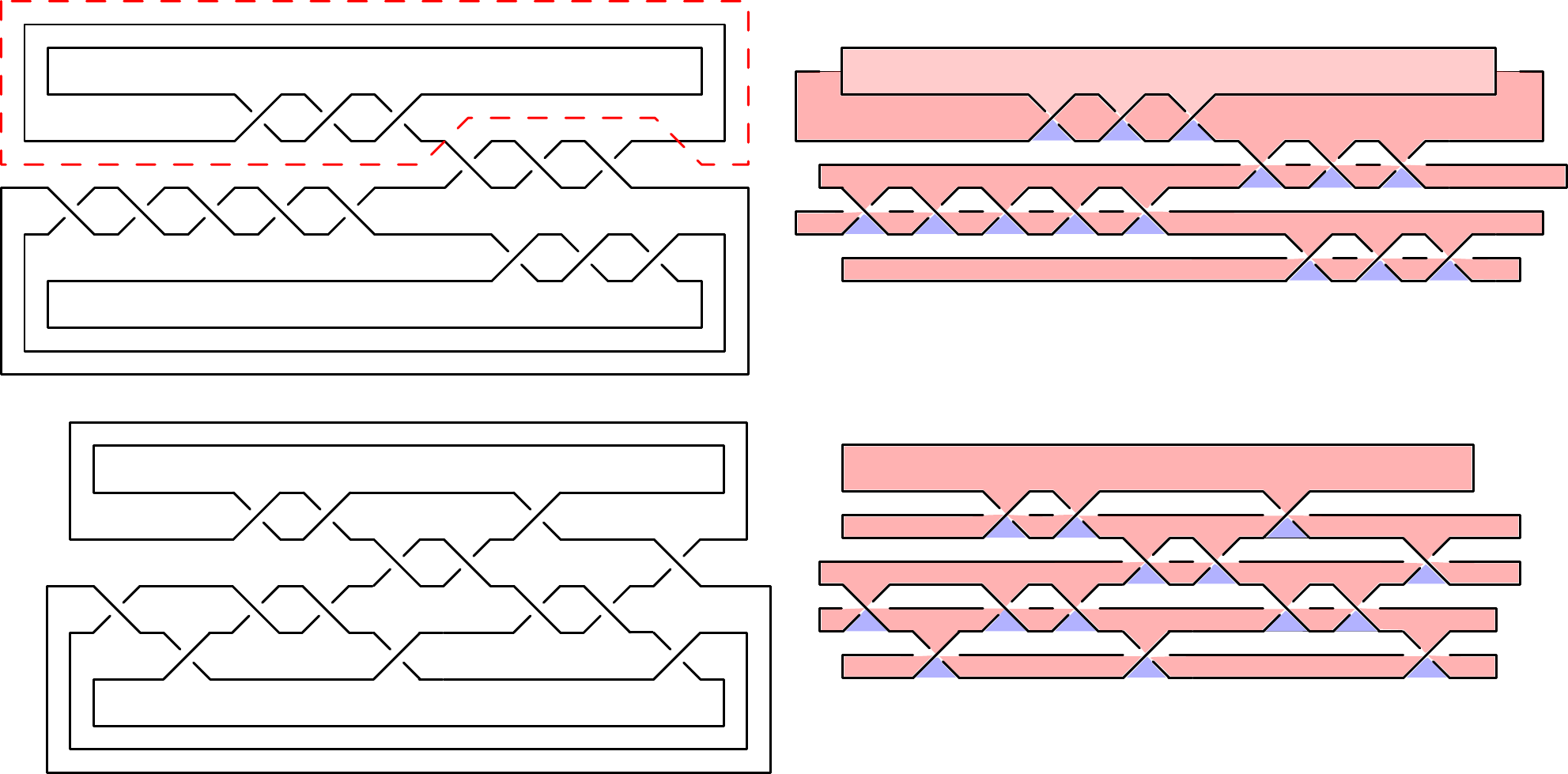}
    \caption{Left: $D(\beta_{\mathrm{comp}})$ (top) with a decomposition circle (dotted, red) and $D(\beta_{\mathrm{prime}})$ (bottom). Right: Corresponding Seifert surfaces.}
    \label{fig:braidanddiagr}
    \end{figure}
The diagram $D(\beta_{\mathrm{prime}})$ does not feature any decomposition circles, while $D(\beta_{\mathrm{comp}})$ does.

The Seifert surface $\Sigma(\beta)$ obtained from Seifert's algorithm applied to $D(\beta)$ consists of a disk for each Seifert circle, and a band for each generator~$\sigma_i^{\pm}$; see right-hand side of \Cref{fig:braidanddiagr} for an illustration for $\beta\in\{\beta_{\mathrm{comp}},\beta_{\mathrm{prime}}\}$.

If the homogeneous braid $\beta$ features each generator $\sigma_i^{\pm}$ at least once, then $\Sigma(\beta)$ is a fiber surface. This is the case for
$\beta_{\mathrm{prime}}$ and~$\beta_{\mathrm{comp}}$. In fact, it turns out that it suffices to consider homogeneous braids~$\beta$ that are given by braid words in which each generator $\sigma_i^{\pm}$ appears at least twice, which we assume in the rest of this sketch.

\subsection*{Step 2: Decomposition as iterative Murasugi sum}

The Seifert surface~$\Sigma(\beta)$ arises as iterative Murasugi sums of fiber surfaces of positive or negative braids (those only featuring  generators of one sign). We describe two such iterative procedures.

Firstly, $\Sigma(\beta)$ is the iterative Murasugi summing of $n-1$ open books with pages that are the fiber surfaces $\Sigma(2,k_1),\dots \Sigma(2,k_{n-1})$ of the torus links~$T(2,k_i)$, where~$k_i\in\mathbb{Z}\setminus\{-1,0,1\}$.
In the case of our two example braids, we find $k_1=k_4=-3$, $k_2=3$, and~$k_3=5$; see \Cref{fig:murasugisforbetaprandbetadec}.
\begin{figure}[ht]
    \centering
    \includegraphics[width=1\linewidth]{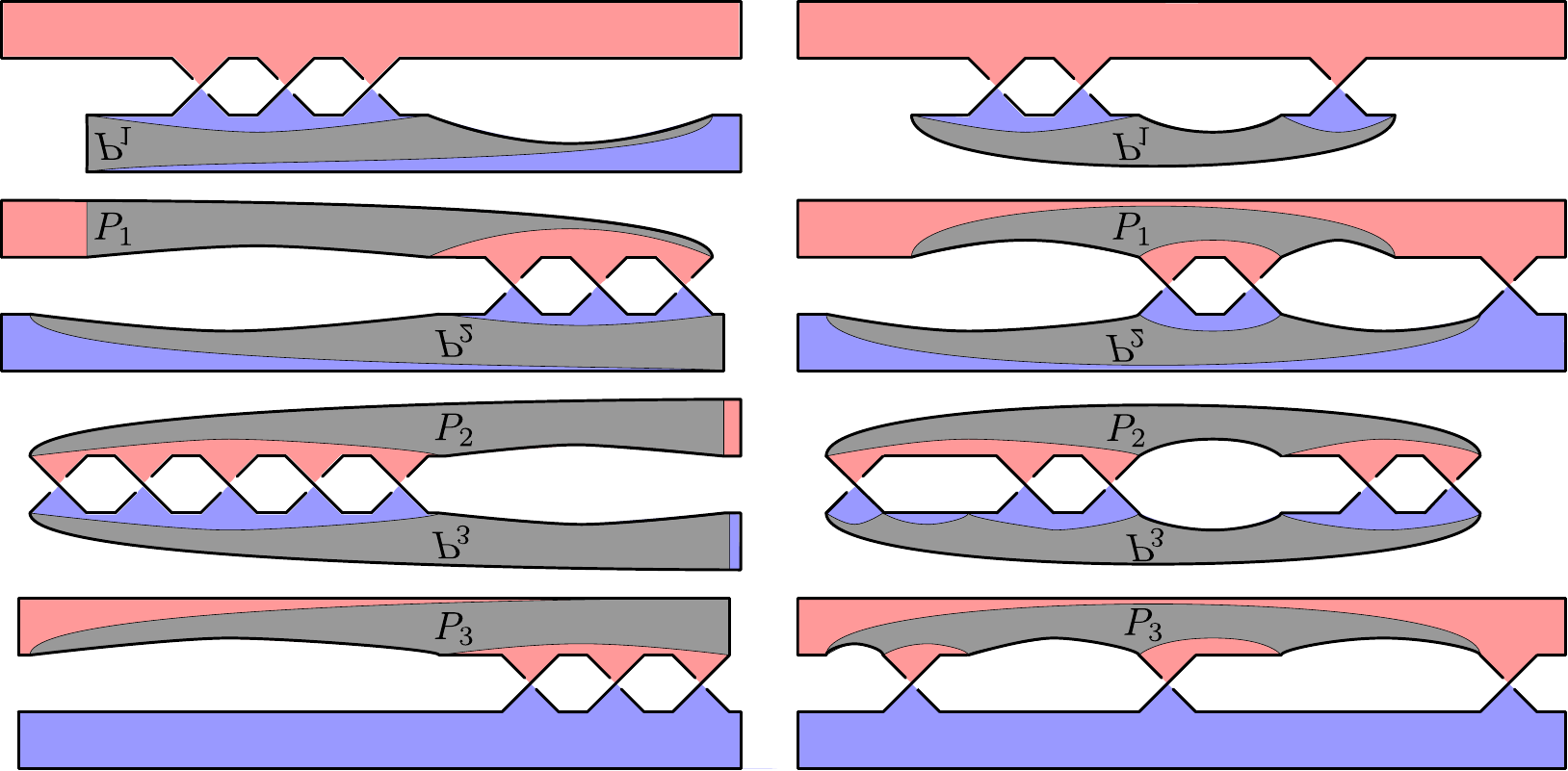}
    \caption{The fiber surfaces $\Sigma_1=F(2,-3)$, $\Sigma_2=F(2,3)$, $\Sigma_3=F(2,5)$, $\Sigma_4=F(2,-3)$ (from top to bottom), together with summing regions $P_1$, $P_2$, and $P_3$ such that Murasugi summing $\Sigma_2$ to $\Sigma_1$ using~$P_1$, summing $\Sigma_3$ to the result via~$P_2$, and summing $\Sigma_4$ to that result via $P_3$ yields $\Sigma(\beta_{\mathrm{comp}})$ (left) and $\Sigma(\beta_{\mathrm{prime}})$ (right).}
\label{fig:murasugisforbetaprandbetadec}
    \end{figure}
It turns out that this iterated Murasugi sum description can be chosen to be an essential Murasugi sum in every step if and only if $D(\beta)$ has no decomposition circles; see \Cref{lemma:braid_murasugi}.
The summing regions for $\Sigma(\beta_{\mathrm{prime}})$ indicated in 
\Cref{fig:murasugisforbetaprandbetadec} are essential, while those for $\Sigma(\beta_{\mathrm{comp}})$ are not.

Secondly, we describe a summing procedure that uses the same summing regions, but first iteratively sums together $\Sigma(2,k_{i})$, $\Sigma(2,k_{i+1})$,$\dots$, $\Sigma(2,k_{i+l})$, where $i$ is chosen minimally and $l$ is chosen maximally such that $k_{i},\dots,k_{i+l}$ have the same sign, yielding open books given by fiber surfaces $\Sigma_1,\dots,\Sigma_m$ for some~$m\leq n-1$. If the summing regions that are used define essential Murasugi sums, then these open books are all either strictly right-veering or strictly left-veering since they are iterative essential Murasugi sums of open books that all are strictly right-veering or all are strictly left-veering; see \Cref{lemma:prime_trees}. In fact, $\Sigma_1,\dots \Sigma_m$ alternate between being strictly right-veering and strictly left-veering.
In the case of~$\beta_{\mathrm{prime}}$, this yields the $3$ fiber surfaces $\Sigma_1$, $\Sigma_2$, and $\Sigma_3$ depicted in \Cref{fig:primsecMurasugisum}. 
\begin{figure}[ht]
    \centering
    \includegraphics[width=0.8\linewidth]{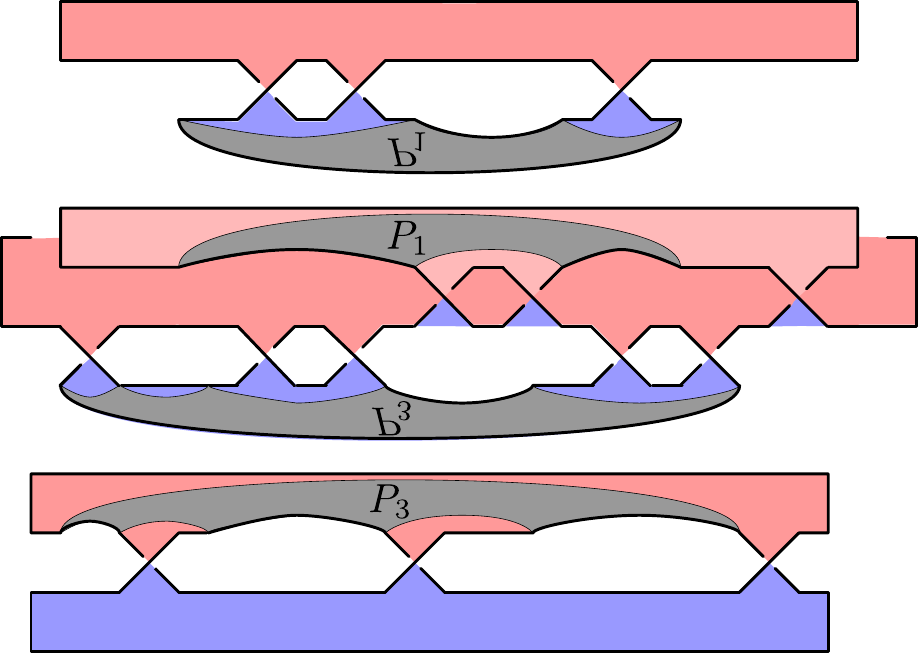}
    \caption{The fiber surfaces $\Sigma_1=F(2,-3)$ (top), $\Sigma_2$ (middle), $\Sigma_3=F(2,-3)$ (bottom), where $\Sigma_2$ arose by Murasugi summing $F(2,5)$ to $F(2,3)$ using the summing region $P_2$ from \Cref{fig:murasugisforbetaprandbetadec}.}
\label{fig:primsecMurasugisum}
    \end{figure}
The corresponding open books are strictly left-veering (top), strictly right-veering (middle), and strictly left-veering (bottom).
The remaining summing regions can now be used to
recover $\Sigma(\beta)$ by iteratively Murasugi summing~$\Sigma_1,\dots,\Sigma_m$.
\subsection*{Step 3: Primeness of the closure of \texorpdfstring{$\beta$}{β}}
Assume $\beta$ is a homogeneous braid such that $D(\beta)$ has no decomposition circles. 
To establish \Cref{thm:main}, we show that the closure of $\beta$---the fibered link $L$ described by $D(\beta)$---is prime.
To see this, we consider $\Sigma(\beta)$ as the iterated Murasugi sum of~$\Sigma_1,\dots,\Sigma_m$.
The $\Sigma_i$ correspond to strictly left-veering or strictly right-veering open books, the iterative Murasugi sums are essential (as there are no decomposition circles), and the veeringness of consecutive open books alternates.
Thus, $\Sigma(\beta)$ arises as an iterative Murasugi sum for which \Cref{prop:nofixedarccriterion} applies in each step.
In case of~$\beta_{\mathrm{prime}}$, we take the Murasugi sum of the open book corresponding to $\Sigma_1$ with the open book corresponding to $\Sigma_2$ using~$P_1$.  Since the open book corresponding to $\Sigma_1$ is strictly left-veering, it has no fixed essential arcs and all arcs in $P_1\subset \Sigma_1$ are mapped to the left by the monodromy; hence, using that $\Sigma_2$ is strictly right-veering, $\Sigma_1\#_{P_1}\Sigma_2$ has no fixed essential arcs by \Cref{prop:nofixedarccriterion}. Noting that $P_3\subset (\Sigma_1\#_{P_1}\Sigma_2)$
lies in~$\Sigma_2$, we have that all arcs in $P_3$ are mapped to the right by the monodromy of~$\Sigma_1\#_{P_1}\Sigma_2$, hence $\Sigma(\beta_{\mathrm{prime}})=(\Sigma_1\#_{P_1}\Sigma_2)\#_{P_3}\Sigma_3$ has no fixed essential arcs by \Cref{prop:nofixedarccriterion} (applied with the words left and right exchanged).

All in all, the monodromy of $\Sigma(\beta)$ has no fixed essential arcs. Equivalently its binding---considered as the fibered link $L=\partial \Sigma(\beta)$---is prime, as desired.

\section{Fixed arcs in Murasugi sums of open books} \label{sec:murasugi_sum}
\subsection{Preliminaries on open books, fixed arcs and primeness}
Recall that an \emph{open book} is a pair~$(\Sigma, \phi)$, where~$\Sigma$, called the \emph{page}, is a compact oriented surface with non-empty boundary, and~$\phi$, called the \emph{monodromy}, is an orientation-preserving self-diffeomorphism of~$\Sigma$ that fixes $\partial\Sigma$ pointwise.

Given an open book~$(\Sigma, \phi)$, we say an arc~$\gamma$, i.e.~an oriented properly embedded closed interval, is \emph{fixed%
} if $\phi(\gamma)$ is isotopic to $\gamma$ relative to its endpoints. We call a properly embedded arc $\gamma$ in $\Sigma$ \emph{essential}, if it is not boundary parallel; that is, $\gamma$ is not isotopic relative endpoints to an arc in~$\partial \Sigma$.
Note that boundary parallel arcs are always fixed. In this text we are concerned with operations such as stabilization and Murasugi sum and when they preserve the property of having no fixed essential arcs. This relates to primeness as follows.

We call an open book $(\Sigma,\phi)$ \emph{prime} if there does not exist an separating essential arc that is fixed by~$\phi$.
Equivalently, $(\Sigma,\phi)$ is prime if there does not exist an embedded $2$--sphere in the $3$--manifold associated with the open book that transversely intersects the page in a separating essential arc. 
Here, the $3$--\emph{manifold $M_\phi$ associated with the open book} $(\Sigma,\phi)$ is the closed oriented $3$--manifold resulting as a quotient of the mapping torus of $\phi$ as follows: collapse the boundary tori of the mapping torus $\Sigma\times [0,1] / (x,0) \sim (\phi(x),1)$ to circles---called the \emph{binding}---by collapsing the $[0,1]$--direction to a point.
The page is considered as a surface in~$M_\phi$, a.k.a. a fiber surface for the binding, via the embedding of $\Sigma$ induced by~$x\mapsto (x,0)$.
In fact, a fixed essential arc is separating exactly if the corresponding $2$--sphere separates the $3$--manifold. In this text we focus on the question of whether an open book has no fixed essential arcs, as this is what we can naturally track from the point of view of open books, rather than separating arcs. While this is a priori a stronger condition than primeness, note that in case the resulting $3$--manifold $M_\phi$ does not feature an embedding of a $2$--sphere that is non-separating, which is equivalent to $M$ not having a $S^1\times S^2$ summand in its prime decomposition, the two notions agree.

Now, let $L$ be a link in an oriented closed $3$--manifold~$M$. We say $L$ is \emph{fibered} if its complement fibers over the circle, that is, there exists a fibration $ \pi\colon M \setminus L \rightarrow S^1$ where the fibers $F_{\theta} = \pi^{-1}(\theta)$ are the interior of diffeomorphic compact surfaces with common boundary~$L$. 
Equivalently, there exists an open book such that $M_\phi$ is diffeomorphic to $M$ by an orientation preserving diffeomorphism that orientation preservingly maps the binding to~$L$. In this case, the page can be taken to be a surface whose interior is diffeomorphic to the fiber~$F_{\theta}$, and the monodromy is the (equivalence class of the) return map of the fibration.
Here, equivalence of open books is the equivalence relation induced by boundary preserving isotopy and conjugation: open books $(\Sigma_1,\phi_1)$ and  $(\Sigma_2,\phi_2)$ are \emph{conjugate} if there exist an orientation preserving diffeomorphism $f\colon \Sigma_1\to\Sigma_2$ such that $\phi_2\circ f=f\circ\phi_1$. Associating to an open book $(\Sigma,\phi)$ the $3$--manifold $M_\phi$  and its binding induces a bijection between open books up to equivalence and fibered links $L$ in $3$--manifolds $M$ up to orientation preserving diffeomorphisms of pairs~$(M,L)$. 
We note that a fibered link in $S^3$ is prime (as defined in knot theory; compare \Cref{sec:braids}) if and only if the corresponding open book is prime, justifying the above definition of primeness for open books.

\subsection{Murasugi sums and representing fixed essential arcs therein}
Originally, the Murasugi sum was understood as an operation on surfaces embedded in $S^3$~\cite{Murasugi_58,Murasugi_63}. However, as Gabai proved that it preserves fiberedness~\cite%
{MurasugiSum}, one can define it for open books of~$S^3$, and the definition extends to open books without restriction on the associated $3$--manifolds.

\begin{definition}\label{def:murasugissum}
    Let $(\Sigma_1, \phi_1)$ and $(\Sigma_2, \phi_2)$ be two open books. 
    Let $n\geq1$ be an integer and $P$ be a $2n$--gon whose boundary has sides $s_1, \dots, s_{2n}$. 
    For~$i = 1,2$, let $f_i$ be an embedding of $P$ into $\Sigma_i$ such that, for~$j = 1, \dots , n$, $f_1(s_{2j-1}) \subset \partial \Sigma_1$, $f_2(s_{2j}) \subset \partial \Sigma_2$, and $f_1(s_{2j})\subset \Sigma_1$ and $f_2(s_{2j-1})\subset \Sigma_2$ are properly embedded arcs in the respective surfaces. We identify $P$ with its image in $\Sigma_1$, $\Sigma_2$, and~$\Sigma$, and often suppress $f_i$ in the notation.
    Let $\Sigma$ be the result of gluing $\Sigma_2$ to $\Sigma_1$ along~$P$, i.e.~$\Sigma\coloneqq\Sigma_1 \cup_{f_1\circ f_2^{-1}\colon f_2(P)\to f_1(P)} \Sigma_2$, and set $\phi\coloneqq \phi_2 \circ \phi_1$. Here, in slight abuse of notation, we write $\phi_1$ to be the diffeomorphism of $\Sigma$ that is obtained by doing $\phi_1$ on $\Sigma_1$ and the identity on $\Sigma_2$ (and similarly for~$\phi_2$).
    The open book $(\Sigma, \phi)$ is called the \emph{Murasugi sum of $(\Sigma_1, \phi_1)$ and $(\Sigma_2, \phi_2)$ along $P$}. Moreover, if for  $j = 1, \dots, n $, $f_1(s_{2j})\subset \Sigma_1$ and $f_2(s_{2j-1})\subset \Sigma_2$ are essential arcs, we say the Murasugi sum and the summing region are \emph{essential}. In another abuse of notation, we denote the sides of~$P \subset \Sigma$ by~$s_i$. Using this notation, a Murasugi sum respectively the plumbing region $P$ is essential, if all the sides $s_i$ of $P\subset\Sigma$ are essential in~$\Sigma$.
\end{definition}

If an $n$--Murasugi sum (i.e.~a Murasugi sum along a $2n$--gon) with $n\geq 2$ is not essential, then it is equivalent to an $(n-1)$--Murasugi sum. Note that a $1$--Murasugi sum is a boundary connected sum. In this case, if neither of the summands is a disk, we can always find a fixed essential arc (the one that separates the two surfaces). Thus we can restrict ourselves to essential Murasugi sums.

\begin{remark}
    When $P$ is a rectangle, i.e.~$n = 2$, we refer to the Murasugi sum as a \emph{plumbing}, if in addition the second open book represents a Hopf link, we refer to it as a \emph{Hopf plumbing} or \emph{stabilization}.
\end{remark}

\begin{remark}\label{rmk:visualMS}
Under the correspondence between open books and fibered links in $3$--manifolds, the Murasugi sum of $L_1$ in $M_1$ with $L_2$ in $M_2$ amounts to the $3$--manifold $M$ given as the connected sum of $M_1$ and $M_2$ and the link $L$ in $M$ defined as follows~\cite{MurasugiSum}. Choose pages in $M_1$ and $M_2$ and closed balls $B_1$ and $B_2$ in $M_1$ and~$M_2$, respectively, such that $B_i$ intersects the page exactly in the summing region~$P$.
Writing $M_i^\circ\coloneqq M_i\setminus \mathrm{int}(B_i)$, we take $M$ to be $M_1\cup_{F}  M_2$, where $F$ is an orientation reversing  diffeomorphism identifying the boundaries of $B_i$ such that it extends the identifications of $\partial P$ viewed in the two pages respectively. Then the two-sphere $S$ in $M$ defined as the quotient of $\partial B_i$ is the decomposition sphere for the connected sum of $M$ and the link $L$ is taken to be (a smoothing of) the result of taking the quotient of
$L_1\setminus \mathrm{int}(B_1)\cup L_2\setminus \mathrm{int}(B_1)\subset M_1\sqcup M_2$ in~$M$. See \Cref{fig:murasugi_example} for a visualization.
\begin{figure}[ht]
    \centering
    \includegraphics[width=0.6\linewidth]{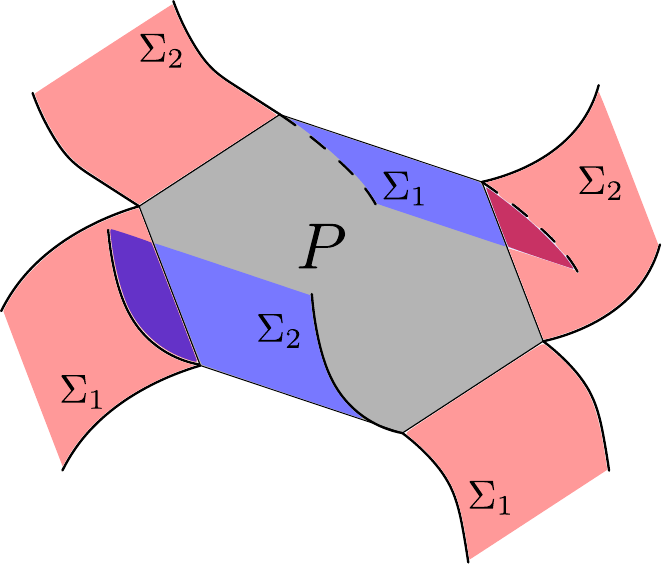}
    \caption{Model of $\Sigma$ in a ball-neighborhood of $P$ in the $3$--manifold associated to a $3$--Murasugi sum. The $6$--gon $P$ along which we are summing is shaded in gray, and lies on a  plane~$E$. The page $\Sigma_1$ lies below~$E$, and above it, lies the page~$\Sigma_2$. In the notation of \Cref{rmk:visualMS}, in case $M_1\cong M_2\cong S^3$, up to diffeomorphism $M$ can be identified with the one point compactification of $\R^3$ such that $S^2$ is the one-point compactification of~$E$, $M_1\setminus \mathrm{int}(B_1)$ and $M_2\setminus \mathrm{int}(B_2)$ are the lower and upper half-spaces with boundary~$E$, respectively, and the pages $\Sigma_i\subset M$ are surfaces in the said half-spaces.}
    \label{fig:murasugi_example}
    \end{figure}
\end{remark}    

\begin{lemma} \label{lemma:lemma1}
    Let $(\Sigma, \phi)$ be the essential Murasugi sum of $(\Sigma_1, \phi_1)$ and $(\Sigma_2, \phi_2)$ along a polygon~$P$. Assume there is an essential arc $\gamma \subset \Sigma$ fixed by~$\phi=\phi_2 \circ \phi_1$. Then one may isotope $\phi_1, \phi_2$,~$\gamma$,
with an isotopy of $\phi_i$ supported in~$\Sigma_i$,
such that $\phi_1(\gamma) = \phi_2^{-1}(\gamma)$ on~$\Sigma$, and $\phi_1(\gamma) = \phi_2^{-1}(\gamma) = \gamma$ on~$\Sigma \setminus P$. 
\end{lemma}
Crucially, $\gamma$ and $\phi_1(\gamma) = \phi_2^{-1}(\gamma)$ can intersect $P$ differently, and, in fact, this has to happen, whenever $(\Sigma_1, \phi_1)$ and $(\Sigma_2, \phi_2)$ do not have fixed essential arcs.
We illustrate \Cref{lemma:lemma1} with open books corresponding to fibered links in~$S^3$. 
\begin{example}\label{ex:6_2}
Let $\Sigma$ be the genus $2$ surface with $2$ boundary components depicted in \Cref{fig:6_2}. We let $\phi$ be the monodromy given by $\tau_5^{-1}\tau_4^{-1}\tau_3\tau_2\tau_1$, where $\tau_i$ denotes the positive Dehn twist about the curve labeled~$i$. 
\begin{figure}[ht]
    \centering
    \includegraphics[width=0.95\linewidth]{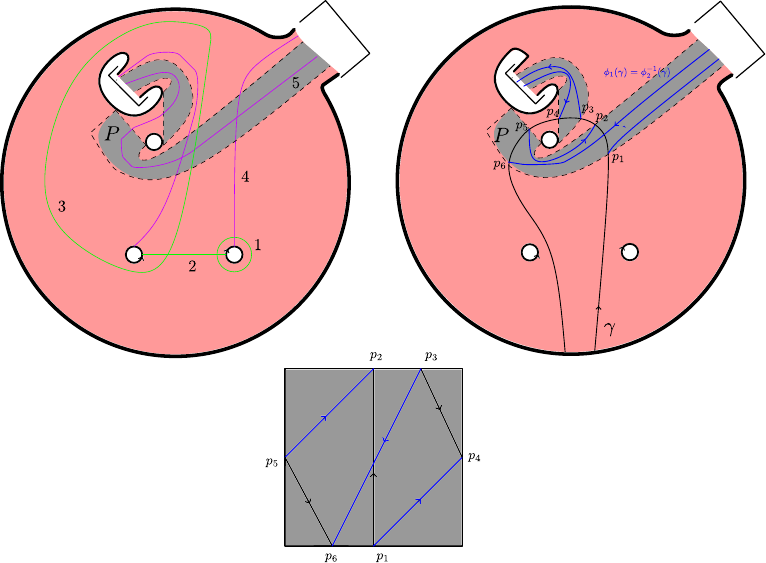}
    \caption{Left: the surface $\Sigma$ with the curves (green, labeled 1, 2, 3; and pink, labeled 4 and 5) of the Dehn twists factors of~$\phi$, and the plumbing region $P$ (grey). Explicitly, $\Sigma$ arises from the depicted plane surface (red) by identifying the oriented circles with an arrow head marking and by completing with a one-handle where indicated with square brackets. The curves and the plumbing regions are completed accordingly.
    \newline
    Right: the arc $\gamma$ (black) is fixed by~$\phi$, and $\tau_5(\gamma) = \phi_2^{-1}(\gamma) = \phi_1(\gamma)$ (blue in~$P$, black in $\Sigma\setminus P$) isotoped to coincide with $\gamma$ outside of~$P$.
    \newline
    Bottom: $\gamma \cap P$ (black), and $\phi_1(\gamma) \cap P = \phi_2^{-1}(\gamma) \cap P $ (blue).}
    \label{fig:6_2}
    \end{figure}
Since $\tau_5^{-1}$ is the last Dehn twist, $(\Sigma,\phi)$ is equivalent to a 2--Murasugi sum (i.e.~plumbing) of~$(\Sigma_1, \phi_1)$, where $\Sigma_1$ is a genus 2 surface with connected boundary, and $\phi_1 = \tau_4^{-1}\tau_3\tau_2\tau_1$ (after identifying $\Sigma_1$ with a subsurface of $\Sigma$);  and $(\Sigma_2, \phi_2)$ is an annulus with a negative Dehn twist. In the notation of \Cref{lemma:lemma1}, extending $\phi_i$ to~$\Sigma$, we have that $\phi=\phi_2\circ\phi_1$, where $\phi_1 = \tau_4^{-1}\tau_3\tau_2\tau_1$ and~$\phi_2=\tau_5^{-1}$.  It is also immediate that the plumbing is essential.
The open book $(\Sigma, \phi)$ has a fixed essential arc~$\gamma$; see right of \Cref{fig:6_2}. We can check that $\gamma$ is fixed by noting that $\tau_5(\gamma) = \phi_2^{-1}(\gamma)$ is isotopic to~$\phi_1(\gamma)$.

By \Cref{lemma:lemma1}, we know that we can arrange (by isotopies fixing the boundary) that $\phi_1(\gamma)$ and  $\phi_2^{-1}(\gamma)$ coincide, and furthermore that $\delta\coloneqq\phi_1(\gamma)=\phi_2^{-1}(\gamma)$ coincides with $\gamma$ outside of~$P$. This is realized on the right of \Cref{fig:6_2}. In this case, $\gamma$ (and hence
$\delta$) intersects $\partial P$ in $6$ points~$p_1, \dots, p_6$; in other words both $\gamma \cap P$ and $\delta \cap P$ consist of $3$ arcs with the same boundary points; see the bottom of \Cref{fig:6_2}. Note that the two sets of 3 arcs $\delta \cap P$ and $\gamma \cap P$ are not isotopic relative $\gamma \cap \partial P=\delta \cap \partial P$. In fact, this must be the case as neither  $(\Sigma_2,\phi_2)$ nor $(\Sigma_1,\phi_1)$ feature fixed essential arcs. To argue the latter, one may for example identify its binding with a prime fibered knot. In fact, one finds the following.

The $3$--manifold associated with $(\Sigma, \phi)$ is diffeomorphic to $S^3$ since $(\Sigma, \phi)$ is equivalent to the open book constructed as follows. Start with an annulus with a positive Dehn twist (corresponding to a positive Hopf band in $S^3$), this twist corresponds to~$\tau_1$. Then do the positive stabilizations corresponding to $\tau_2$ and~$\tau_3$, and finally the negative stabilization corresponding to $\tau_4^{-1}$ to obtain~$(\Sigma_1, \phi_1)$. It turns out that the binding of $(\Sigma_1, \phi_1)$ is the prime fibered knot~$6_2$, while the binding of~$(\Sigma, \phi)$, the result after the final plumbing with~$(\Sigma_2, \phi_2)$, is the connected sum of the prime fibered knot $8_{21}$ and a negative Hopf link.
Because of this, we think of this example as a sibling to \Cref{ex:counterexample} from the introduction, and invite the reader to explicitly realize the pages and the binding in $S^3$ to check that indeed the resulting bindings are $6_2$ and the connected sum of a Hopf link with~$8_{21}$. Vice versa, we invite the reader to realize \Cref{ex:counterexample} as an open book with monodromy $\phi=\phi_2\circ\phi_1$ given by a product of Dehn twists and to realize $\phi_1$ and $\phi_2$ such that $\phi_1(\gamma)$ and $\phi_2^{-1}(\gamma)$ coincide and only differ from $\gamma$ in the plumbing region~$P$.
\end{example}
\begin{proof}[Proof of \Cref{lemma:lemma1}]
    First, we may and do assume that $\partial \gamma$ is disjoint from the corners of $P$ (if it is not, isotope $\partial \gamma$ away along the boundary, apply the lemma, and isotope back).
    
    Let $s_1, \dots , s_{2n}$ be the sides of~$P$. These are essential arcs in~$\Sigma$. Isotope $\gamma$ relative to the boundary such that it is simultaneously transverse to all the $s_i$ and intersects them minimally.
    Note that $\phi_1(\gamma)$ has no bigons with $s_i$ if~$s_i \subset \partial \Sigma_1$. However, for $s_i \subset \partial \Sigma_2$, $\phi_1(\gamma)$ and $s_i$ may have bigons. In this case, we can isotope $\phi_1$ with an isotopy supported in $\Sigma_1$ such that $\phi_1(\gamma) \eqqcolon \delta_1$ and $s_i$ have no bigons. Then $\delta_1$ intersects all $s_i$ minimally.
    Similarly we may isotope $\phi_2$ with an isotopy supported in $\Sigma_2$ such that $\phi_2^{-1}(\gamma) \eqqcolon \delta_2$ intersects all $s_i$ minimally. 

    Observe that for~$s_i \subset \partial \Sigma_1$, we have equalities of intersection numbers
$i(\gamma, s_i) = i(\phi_1(\gamma), \phi_1(s_i)) = i(\delta_1, s_i) = i(\delta_2, s_i)$, and similarly for~$s_i \subset \partial \Sigma_2$.
    Hence, we may isotope $\phi_1$ such that $\delta_1 \cap s_i = \delta_2 \cap s_i$ for all $s_i \subset \partial \Sigma_2$ and analogously isotope $\phi_2$ such that $\delta_1 \cap s_i = \delta_2 \cap s_i$ for all~$s_i \subset \partial \Sigma_1$.

    Overall, we have that $\delta_1$ and $\delta_2$ are isotopic, they both intersect all $s_i$ minimally, and their intersections with $s_i$ coincide for all~$i$. By \Cref{lemma:lemma2} below, there exists an isotopy of $\Sigma$ sending $\delta_1$ to~$\delta_2$, fixing all $s_i$ pointwise.
This isotopy can be written as composition of an isotopy with support in $\Sigma_1$ and an isotopy with support in~$\Sigma_2\setminus P$.
We now change $\phi_1$ and $\phi_2$ by composing with the first and second of those isotopies, respectively.
       In this way, we have achieved~$\delta_1 = \delta_2$. Moreover, observe the following. Since $\phi_1$ is the identity on~$\Sigma \setminus \Sigma_1$, we have that $\delta_1 = \gamma$ on~$\Sigma \setminus \Sigma_1$. Similarly $\delta_2 = \gamma$ on~$\Sigma \setminus \Sigma_2$. This implies $\delta_1 = \delta_2 = \gamma$ on $\Sigma\setminus P$ as desired.
\end{proof}

\begin{lemma} \label{lemma:lemma2}
    Let $\Sigma$ be a compact surface, $s_1, \dots, s_n$ properly embedded arcs with pairwise disjoint interiors, and $\delta_1, \delta_2$ isotopic properly embedded arcs with ${\partial \delta_1 = \partial \delta_2}$ and such that they each intersect each $s_i$ minimally. Moreover, assume that $\delta_1 \cap s_i = \delta_2 \cap s_i$ for all i. Then, there exists an isotopy of $\Sigma$ fixing each $s_i$ pointwise that sends $\delta_1$ to~$\delta_2$. 
\end{lemma} 

\begin{proof}

    If $\delta_1$ and $\delta_2$ are isotopic, then the bigon-criterion states that an isotopy sending $\delta_1$ to $\delta_2$ is realized by removing bigons cut out by them. %
    Suppose first that an isotopy sending $\delta_1$ to $\delta_2$ can be done without removing any vertices of the bigons. 
    Then in particular it fixes the intersection points $x \in \delta_1 \cap \delta_2 \cap s_i$, and since the intersection points of both arcs with all $s_i$ coincide, the bigons are disjoint from $s_i$ except possibly at their vertices (which are fixed), and so the isotopy fixes all $s_i$ pointwise; see the left hand side of \Cref{fig:allowed_isotopy}.
\begin{figure}[tb]
    \centering
    \includegraphics[width=0.65\linewidth]{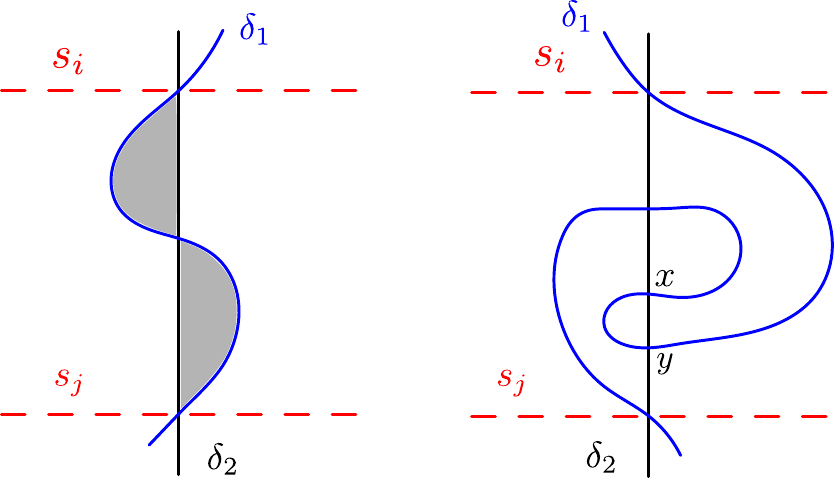}
    \caption{On the left, we can remove the shaded bigons by an isotopy fixing the vertices, in particular, this isotopy fixed $s_i$ and $s_j$ pointwise. On the right, we can remove the intersection points $x$ and $y$ by an isotopy fixing $s_i$ and~$s_j$, and then we are in the case of the left hand side.}
    \label{fig:allowed_isotopy}
    \end{figure}
\begin{figure}[tb]
    \centering
    \includegraphics[width=0.65\linewidth]{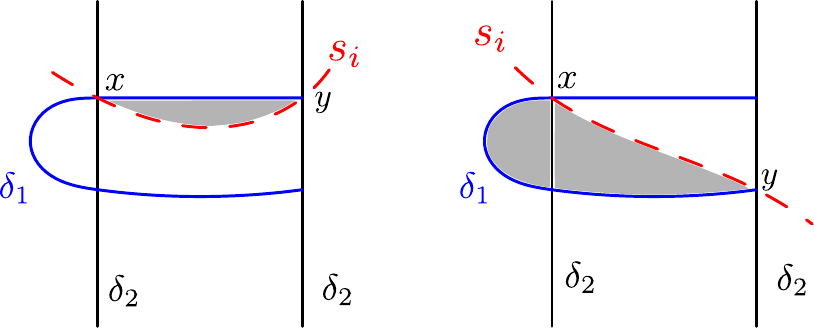}
    \caption{If the isotopy removes a point $x \in \delta_1 \cap \delta_2 \cap s_i$, then we can find a bigon that shows that the intersection of $\delta_1$ and $s_i$ was not minimal.}
    \label{fig:impossible_isotopy}
    \end{figure}

    Now suppose that the isotopy removes pairs of intersection points. If all of them lie away from the~$s_i$, then the isotopy removing them fixes~$s_i$; see the right hand side of \Cref{fig:allowed_isotopy}. So assume one of these intersection points $x$ lies on some~$s_i$. Note that every bigon that removes the intersection point has to be contained in a bigger bigon (since $\delta_1$ and $\delta_2$ are isotopic). Then, since the intersection is transverse, $s_i$ must enter this bigon. However, it must also leave it by another intersection $y \in \delta_1 \cap \delta_2$. This implies that $s_i$ forms a bigon with one of the arcs (say~$\delta_1$) that can be removed, reducing the number of intersection points~$\delta_1 \cap s_i$, contradicting minimality; see \Cref{fig:impossible_isotopy}.
    \end{proof}

\subsection{Right-veering and left-veering for arcs and monodromies, and veering in polygons}
We say an arc $\gamma$ in a surface $\Sigma$ \emph{veers to the right} under a diffeomorphism $\phi$ of~$\Sigma$, if
$\gamma$ is not fixed,
and may be isotoped such that it intersects its image transversely and minimally
and at the starting point the tangent vector of $\phi(\gamma)$ followed by the tangent vector of $\gamma$ yield the orientation of~$\Sigma$. We denote this by~$\phi(\gamma) > \gamma$. Analogously, we have the notion of $\gamma$ \emph{veering to the left} under $\phi$ and denote this by~$\phi(\gamma)< \gamma$. We write~$\phi(\gamma)\geq \gamma$, if $\gamma$ veers to the right or is fixed by~$\phi$, and similarly we write  $\phi(\gamma)\leq \gamma$ if $\gamma$ veers to the left or is fixed by~$\phi$. More generally, for arcs $\gamma$ and $\gamma'$ that have the same start point, we write $\gamma<\gamma'$, $\gamma\leq\gamma'$, $\gamma>\gamma'$, and~$\gamma\geq\gamma'$, and mean the analogously defined notion (with $\gamma'$ in place of~$\phi(\gamma)$).

For a given open book, if no essential arc veers to the left, we say that the open book is \emph{right-veering}, if every essential arc veers to the right, we say that the open book is \emph{strictly right-veering}. \emph{Left-veering} and \emph{strictly left-veering} open books are defined analogously.

In the next lemma we consider embedded unoriented arcs in a closed disk~$D$. In fact, we will take $D$ to be a compact convex subset of $\mathbb{R}^2$ and we take the unoriented arcs to be straight arcs, which we call \emph{chords}. If two chords $a$ and $b$ share an endpoint~$p$, we say $a$ is to the right of $b$ at~$p$, if the arcs $\gamma_a$ and $\gamma_b$ given by orienting $a$ and~$b$ with starting point~$p$, respectively, point $p$ satisfy~$\gamma_a > \gamma_b$. 
\begin{lemma} \label{lemma:lemma3}
    Let $D$ be a disk, $A$ and $B$ finite sets of chords with $a \cap a' = b \cap b' = \varnothing$ for $a\neq a' \in A$, $b\neq b' \in B$, and such that $\bigcup_{a \in A} \partial a = \bigcup_{b \in B} \partial b$. 
    Then either~$A = B$, or there exist $a_1, a_2 \in A$, $b_1, b_2 \in B$ with $a_i \cap b_i = \{ p_i \} \in \partial D$, and seen from~$p_1$, $b_1$ is to the right of~$a_1$, and seen from~$p_2$, $b_2$ is to the left of~$a_2$. 

    Moreover, if $D$ is an $n$--gon, and for each chord~$c\in A\cup B$, $c$ has endpoints on different sides, we can find chords as above with the additional condition that for~$i = 1,2$, the other endpoint of $a_i$ is on a different side than the other endpoint of~$b_i$.
\end{lemma}
As an illustration of \Cref{lemma:lemma3}, consider for example the 4--gon $P$ on the bottom of \Cref{fig:6_2}, where $A = \gamma \cap P$ and $B = \phi_1(\gamma) \cap P = \phi_2^{-1}(\gamma) \cap P$. At $p_1$ the chord of $B$ is to the right of the chord of $A$ (and ends on a different side), while at $p_5$ the chord of $B$ is to the left of the chord of $A$ (and also ends on a different side).
\begin{proof}[Proof of \Cref{lemma:lemma3}]
    For the first claim, assume towards a contradiction that~$A \neq B$, and that for all~$a \in A$, $b \in B$ such that $a\cap b = \{p\} \in \partial D$, $b$ is to the right of~$a$ at~$p$. 
    Since~$A \neq B$, there exist such~$a,b$. Now, let $p'$ be the other endpoint of~$b$. Let $a'$ be the chord in $A$ with~$p' \in a'$. Since~$p \notin a'$, $b$ has to be to the right of $a'$ as seen from~$p'$. Let $p''$ be the other endpoint of~$a'$. Let $b' \in B$ be the chord such that~$p'' \in b'$. See~\Cref{fig:zigzag}.
    \begin{figure}[b]
    \centering
    \includegraphics[width=0.3\linewidth]{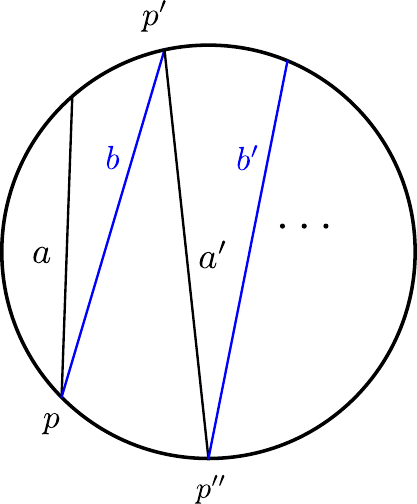}
    \caption{At each endpoint, the corresponding chord from $B$ is to the right of the chord from $A$ by hypothesis, but finiteness of $A$ and $B$ implies that this cannot happen indefinitely.}
    \label{fig:zigzag}
    \end{figure}
    We cannot keep doing this indefinitely because of finiteness of $A$ and~$B$, but this means there would be points in on $\partial D$ belonging only to $\bigcup_{a \in A} \partial a$ or only to $\bigcup_{b \in B} \partial b$. This gives the contradiction, as desired. 

For the second claim, suppose that $D$ is an $n$--gon, and suppose no chord joins points on the same side of~$\partial D$. Assume towards a contradiction that for any $a \in A, b \in B$ such that~$a\cap b = \{p\} \in \partial D$, and $b$ is to the right of~$a$, the other endpoints of $a$ and $b$ lie on the same side of~$D$. Note that a pair $a,b$ where $b$ is to the right of~$a$ at $a\cap b = \{p\} \in \partial D$ exists by the first claim, proved in the previous paragraph.
   
    Then, take one such pair~$a,b$. Let $p'$ be the other endpoint of~$b$, and $a'$ the chord of $A$ with endpoint on~$p'$. Now, seen from~$p'$, $b$ is to the right of $a'$ (if it were not, since~$a \cap a' = \varnothing$, $a'$ would have both its endpoints on the same side of~$D$, see the left hand side of \Cref{fig:different_sides}). However, this means that the other endpoint of~$a'$, call it~$p''$, lies on the same side of $D$ as~$p$. Then, let $b' \in B$ be the chord with endpoint on~$p''$. It must be to the right of $a'$ seen from~$p''$. Hence by the hypothesis, the other endpoint of $b'$ must be on the same side of $D$ as~$p'$. We end up in a situation like in the previous case, see the right hand side of \Cref{fig:different_sides}.
    \begin{figure}[tb]
    \centering
    \includegraphics[width=0.75\linewidth]{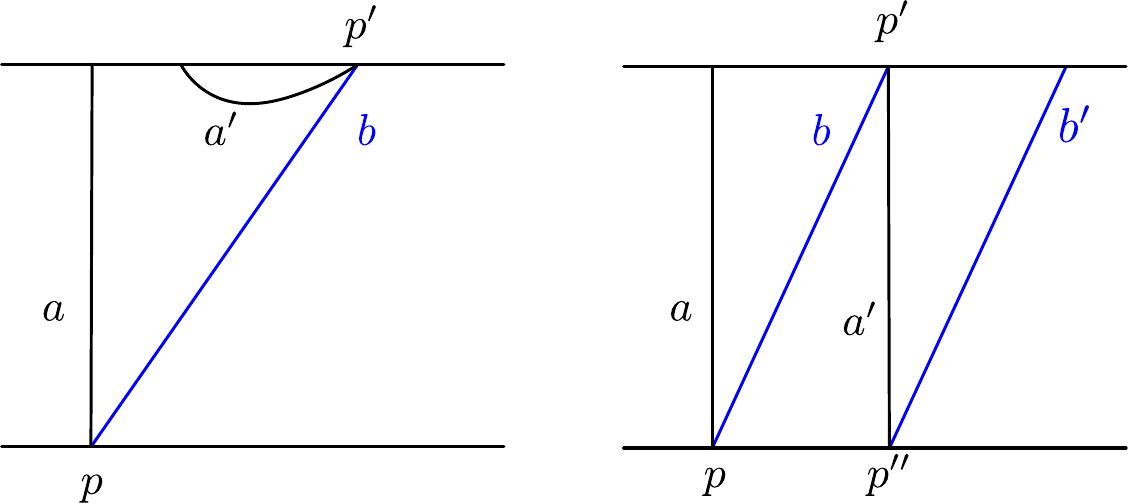}
    \caption{On the left hand side, if $a'$ is to the right of $b$ at~$p'$, then it would have both its endpoints on the same side. On the right hand side, we see the at each endpoint the chords from $B$ must be to the right of the chords from~$A$, but this cannot happen indefinitely.}
    \label{fig:different_sides}
    \end{figure}
    As before, this procedure has to end because of finiteness, leading to a contradiction. Therefore, there do exist $a\in A, b \in B$ with $a\cap b = \{p\}$ and $b$ to the right of $a$ at~$p$, such that the other endpoints of $a$ and $b$ lie on different sides. Observe that the same argument works for $b$ to the left of~$a$.
\end{proof}

\subsection{Proof of \texorpdfstring{\Cref{prop:nofixedarccriterion}}{Proposition \ref{prop:nofixedarccriterion}}}
We are now ready for the proof of \Cref{prop:nofixedarccriterion}, which we restate for the convenience of the reader.
\criterion*
\begin{proof}
    Assume towards a contradiction that there exists a fixed essential arc in~$(\Sigma, \phi)$. Among such arcs, pick $\gamma$ that minimizes the number of intersections points~$\gamma\cap \partial P$.
    By an isotopy of $\gamma$, $\phi_1$ and $\phi_2$ coming from \Cref{lemma:lemma1}, we arrange for $\phi_1(\gamma)=\phi_2^{-1}(\gamma)$ on $\Sigma$ and $\phi_1(\gamma)=\phi_2^{-1}(\gamma) = \gamma$ on~$\Sigma\setminus P$. The minimality of $\gamma\cap \partial P$ still holds true.
    
    If (after an isotopy of $\phi_1$ supported in~$P$) $\phi_1(\gamma) = \gamma$ in all of~$\Sigma$, then we also have $\phi_2^{-1}(\gamma) = \gamma$ in all of $\Sigma$ (after an isotopy of $\phi_2$ supported in~$P$). Consider~$\gamma\cap \Sigma_2$, which is a finite union of proper arcs in~$\Sigma_2$, each of them fixed by~$\phi_2^{-1}$.
    However, $\phi_2^{-1}$ is strictly left-veering (viewed as a map on~$\Sigma_2$), hence each of those arcs is boundary parallel in~$\Sigma_2$. Using that $\gamma$ and $P$ are essential and $\gamma$ and $\partial P$ intersect minimally, one may deduce that $\gamma\cap \Sigma_2 = \varnothing$. Indeed, $\gamma\cap \Sigma_2$ is a collection of boundary parallel arcs in $\Sigma_2$ (that is, together with segments of~$\partial \Sigma_2$, they cut out bigons). Then, using these bigons we may isotope $\gamma$ freely along the boundary away from~$\Sigma_2$.
Thus, we have $\gamma \subset \Sigma_1 \setminus P$, contradicting the assumption that $(\Sigma_1, \phi_1)$ has no fixed essential arcs disjoint from~$P$ (since the Murasugi sum is essential, if $\gamma \subset \Sigma_1$ is essential in $\Sigma$ then it must also be essential in $\Sigma_1$). 

    Therefore, we may and do assume that the arc $\phi_1(\gamma) = \phi_2^{-1}(\gamma)$ does not agree with~$\gamma$, even after any isotopy supported in $P$ (recall that they do agree outside of~$P$). In other words, thinking of $P$ as a polygon in~$\mathbb{R}^2$, the following two set of chords of $P$ differ: the set of chords $A$ given by straightening the components of $\gamma \cap P$ and the set of chords $B$ given by straightening the components of~$\phi_1(\gamma) \cap P$. 
    Observe that $A$ and $B$ satisfy the hypotheses of \Cref{lemma:lemma3}, including the $n$--gon condition because of minimal intersection of $\gamma$ and $\phi_1(\gamma)$ with the sides of~$P$. Hence, \Cref{lemma:lemma3} yields chords~$a\in A$, $b \in B$ such that~$a \cap b = \{p \}$, $b \geq a$ as seen from~$p$, and the other endpoints of $a$ and $b$ (call them $p_a$ and~$p_b$) lie on different sides of~$P$. 
    Let $i\in\{1,2\}$ such that~$p \in \partial \Sigma_i$.
    Let $a^{\mathrm{ex}}$ be the unique subarc of $\gamma$ that starts at $p$ and is a proper arc in~$\Sigma_i$.
    ($a^{\mathrm{ex}}$ extends the arc in $P$ corresponding to~$a$, and may in fact just be equal to it, namely if~$p_a \in \partial\Sigma_i$). Similarly extend $b$ to~$b^{\mathrm{ex}}$. Then $b^{\mathrm{ex}}= \phi_1(a^{\mathrm{ex}})$ or $b^{\mathrm{ex}}= \phi_2^{-1}(a^{\mathrm{ex}})$ if $i = 1$ or~$i = 2$, respectively. Since 
$b^{\mathrm{ex}} > a^{\mathrm{ex}}$ and $\phi_2^{-1}$ is left-veering, it follows that $i = 2$ cannot happen. Therefore, we have~$i = 1$. 

    Since $p_a$ and $p_b$ lie on different sides of~$P$, and since sides of $P$ alternate between $\partial \Sigma_1$ and~$\partial \Sigma_2$, there exists a chord with one endpoint equal to~$p$, and the other on $\partial \Sigma_1$ between $p_a$ and~$p_b$. Let $c$ be the proper arc in $\Sigma_1$ starting at $p$ corresponding to this chord. In particular, $b^{\mathrm{ex}} = \phi_1(a^{\mathrm{ex}}) > c > a^{\mathrm{ex}}$. However, then $\phi_1(c) > \phi_1(a^{\mathrm{ex}}) = b^{\mathrm{ex}} > c$, contradicting $\phi_1(c) \leq c$, which holds by the hypothesis on~$\phi_1$.
\end{proof}

We conclude this section with the remark that from the above proof it is clear that in fact we can make a statement concerning the fixed essential arcs $(\Sigma, \phi)$ as in \Cref{prop:nofixedarccriterion}, even when we drop the assumption that $(\Sigma_1, \phi_1)$ has no fixed essential arc.
\begin{remark}\label{rem:fixarcsdesc}
Let $P$ be an essential summing region for open books $(\Sigma_1, \phi_1)$ and $(\Sigma_2, \phi_2)$ such that $\phi_1(a)\leq a$ for every essential arc $a$ in $P\subset \Sigma_1$ and  $(\Sigma_2, \phi_2)$ is strictly right-veering. Then all fixed essential arcs $\gamma$ of $(\Sigma, \phi)$ can be isotoped (with an isotopy that preserves the boundary setwise) to lie in $\Sigma_1\setminus P\subset \Sigma$. In particular, up to isotopy, all fixed essential arcs of $(\Sigma, \phi)$ arise from fixed essential arcs of $(\Sigma_1, \phi_1)$ that are disjoint from~$P\subset \Sigma_1$. \Cref{prop:nofixedarccriterion} corresponds to the case where we further assume that $(\Sigma_1, \phi_1)$ has no fixed essential arcs.

To see this, just note that in the proof of \Cref{prop:nofixedarccriterion}, one may take $\gamma$ to be any fixed arc that in its isotopy class (preserving the boundary setwise) minimizes the intersection points with~$\partial P$, and rather than reaching a contradiction, conclude that $\gamma \subset \Sigma_1\setminus P$, as is done in the second paragraph of the proof.
\end{remark}

\section{Iterative Murasugi summing of strictly veering open books}
\label{sec:trees}

We define a construction of iterative Murasugi sums, encoded by a tree with vertices open books and edges indicating how they are plumbed together, that yields a class of prime open books, in fact open books without essential fixed arcs. The resulting open books subsume the links that arise from homogeneous braid diagrams without decomposition circles and all fibered arborescent links in~$S^3$.

Let $T$ be a tree, where the vertex set $V$ consists of open books and each edge $e=\{v_1=(\Sigma_1,\phi_1),v_2=(\Sigma_2,\phi_2)\}$ is labeled by an essential summing region $P$ for $v_1$ and~$v_2$, i.e.~embeddings of $f_i\colon P\to \Sigma_i$ as described in Definition~\ref{def:murasugissum}. We further assume that all summing regions are disjoint; more precisely, if $e=\{v_1,v_2\}$ and $e'=\{v_1,v_3\}$ that have a vertex $v_1$ in common, then their labels, say $f_1\colon P\to \Sigma_1$ and $f_2\colon P\to \Sigma_2$ for $e$ and $f_1'\colon P'\to \Sigma_1$ and $f_3\colon P'\to \Sigma_3$ for~$e'$, satisfy~${f_1(P)\cap f_1'(P')=\varnothing}$. We call such a labeled tree a \emph{tree of open books}.

Let $\mathrm{gr}$ be a choice of \emph{growing} of a tree~$T$, i.e.~inductive construction of $T$ starting from a single vertex by iteratively growing leaves. A growing of a tree $T$ yields a word in its vertices (by listing them from right to left in order of appearance).  Given a tree of open books $T$ and a growing
$\mathrm{gr}$ of~$T$, we have an associated open book $(\Sigma,\phi)$ given by iterative Murasugi sums guided by the growing. More precisely, $\Sigma$ is the result of gluing the surfaces using the labels of the edges, and $\phi$ is given as follows. Ordering the vertices according to the growing, relabeling them order preservingly as
\[v_1=(\Sigma_1,\phi_1),v_2=(\Sigma_2,\phi_2),\dots, v_n=(\Sigma_n,\phi_n), \] and writing $\phi_i\colon \Sigma\to \Sigma$ for the extension of~$\phi_i\colon \Sigma_i\to \Sigma_i$, we set \[\phi\coloneqq \phi_n\circ\phi_{n-1}\circ\dots\circ \phi_2\circ \phi_1.\]

\subsection{Figure-eight plumbings}
We illustrate these notions by realizing the concept of a figure-eight plumbing (defined in the introduction) as the associated open book of a tree of open books.

\begin{example}\label{ex:fig8plumb} Let $w=(\Sigma,\phi)$ be an open book and let $a$ be an arc in~$\Sigma$. We describe a tree of open books $T$ with a growing such that the resulting open book is the figure-eight plumbing of $w$ along~$a$. For this let $p$ and $n$ be open books corresponding to a positive and negative Hopf link, respectively, that is, their pages are an annulus, say $A_p$ and~$A_n$, respectively, and their monodromies are a right-veering and left-veering Dehn twist, respectively. Let $T$ be the tree with vertices $w, p, n$ and edges~$\{w,p\}, \{p,n\}$. We pick the unique growing $\mathrm{gr}$ of $T$ that starts with~$w$. We further pick summing regions $P$ and $P'$ as decoration for  $\{w,p\}$ and~$\{p,n\}$, respectively as follows:
both are $4$--gons, $P\subset \Sigma$ is a neighborhood of~$a$, $P,P'\subset A_p$ are disjoint and are neighborhoods of the cocore of $A_p$ and $P'\subset A_n$ is a neighborhood of a cocore of~$A_p$.

Everything is set up such that the open book associated with the tree of open books $T$ and the growing $\emph{gr}$ is the open book resulting from figure-eight plumbing of $(\Sigma,\phi)$ along~$a$.
\end{example}

Before moving on to discussing the main result of the section, namely a sufficient condition that the open books associated with a tree of open books has no essential fixed arcs, we prove that essential figure-eight plumbing preserves having no essential arcs; i.e., we prove \Cref{prop:primnessoffigure8plumbing}.
\begin{proof}[Proof of \Cref{prop:primnessoffigure8plumbing}]
Let $w=(\Sigma,\phi)$ be an open book and let $a$ be an essential arc in $\Sigma$ such that all essential fixed arcs of $(\Sigma,\phi)$ have non-zero intersection with~$a$.
We use the tree of open books $T$ and its growing $\mathrm{gr}$ given in \Cref{ex:fig8plumb} to describe the figure-eight plumbing of $w$ along~$a$.

We denote the result of plumbing $p$ to $w$ by~$(\Sigma_1,\phi_1)$. Now we apply \Cref{prop:nofixedarccriterion} (rather, its analog with left in place of right) to the Murasugi sum of $(\Sigma_1,\phi_1)$ and $n$ along~$P'$, so we check the necessary assumptions. Firstly, the plumbing region $P'$ is essential, hence it is an essential Murasugi sum. Secondly, all essential arcs in $\Sigma_1$ that are contained in $P'\subset \Sigma_1$ go to the right since they are isotopic to a cocore of the annulus which is mapped by the right-veering Dehn twist.
Finally, we remark that, if $(\Sigma_1,\phi_1)$ has fixed essential arcs (which is possible, see e.g.~\Cref{ex:counterexample}), then they all have non-trivial intersection with~$P'$, since all fixed arcs of $w$ run through~$P$. Thus the criterion applies to the open book $(\Sigma_1,\phi_1)$ and the strictly left-veering open book $n$ with summing region~$P'$, and, hence, $(\Sigma_1,\phi_1)\#_{P'}n$ has no essential fixed arcs.

Note that both the construction and the proof carry through if we first plumb the negative Hopf band and then the positive one.
\end{proof}
Using \Cref{rem:fixarcsdesc} in place of \Cref{prop:nofixedarccriterion}, we find that fixed arcs in figure-eight plumbings are the fixed arcs that arise from the original surface.
\begin{remark}\label{rem:fixarcsinfig8} Let $(\Sigma,\phi)$ be an open book and $a$ an essential arc in~$\Sigma$. Then every fixed essential arc in the open book arising from figure-eight plumbing of $(\Sigma,\phi)$ along $a$ can be isotoped (fixing the boundary setwise) into~$\Sigma\setminus P$.

\end{remark}

\subsection{Conjugation of open books and independence of the choice of growing}

Recall that two open books 
$(\Sigma_1,\phi_1)$ and $(\Sigma_2,\phi_2)$ are said to be \emph{conjugate} if there exists an orientation preserving diffeomorphism $f\colon \Sigma_1\to \Sigma_2$ such that~$\phi_2\circ f=f\circ\phi_1$. 

\begin{lemma} Let $T$ be a tree of open books. For every two growings of~$T$, the associated open books are conjugate. 
\end{lemma}

\begin{proof}

Recall the description of the monodromy $\phi$ of the open book associated to a tree $T$ of open books and a growing of $T$ as $\phi\coloneqq \phi_n\circ\phi_{n-1}\circ\dots\circ \phi_1$, where all the $\phi_i$ are diffeomorphisms of $\Sigma$ and the order of appearance depends on the growing. Note that $\phi_i$ and $\phi_j$ commute (in the group of orientation preserving diffeomorphisms of~$\Sigma$), whenever $\{v_i,v_j\}$ is not an edge of~$T$, because then $\phi_i$ and $\phi_j$ have disjoint support.
Thus the result follows from the following elementary fact about group elements.
\end{proof}

\begin{lemma}
Let $x_1,\dots,x_n$ be elements in a group.
Let $T$ be a tree with vertex set $\{x_1,\dots x_n\}$ such that $x_i$ commutes with $x_j$ if $\{x_i,x_j\}$ is not an edge of the tree.
Let $\mathrm{gr}_1$ and $\mathrm{gr}_2$ be two growings of $T$ and let $w_1$ and $w_2$ be the words associated to $\mathrm{gr}_1$ and~$\mathrm{gr}_2$, respectively. Then, then there exists a cyclic permutation of $w_2$ that represents the same group element as~$w_1$, in particular, the elements defined by $w_1$ and $w_2$ are conjugate.\qed
\end{lemma}

\subsection{Primeness result}
Finally we state and prove our main primeness result for trees of open books with strictly veering vertices.
\begin{theorem}\label{thm:prime_trees}
Let $T$ be a tree of open books, where each open book is strictly left-veering or strictly right-veering. Then the associated open book, well-defined up to conjugacy, has no fixed essential arcs; in particular, the fiber surface is a prime surface in the associated $3$--manifold.
\end{theorem}

In the case that all open books are strictly right-veering, the resulting open book is strictly right-veering and thus, in particular, it has no fixed essential arcs.

\begin{lemma}\label{lemma:prime_trees}
If $(\Sigma_1,\phi_1)$ and $(\Sigma_2,\phi_2)$ are strictly right-veering open books, then the result of essential Murasugi sum of the two is strictly right-veering.
More generally,
let $T$ be a tree of open books, where each open book is strictly right-veering. Then the associated open book, well-defined up to conjugacy, is strictly right-veering.
\end{lemma}
\begin{proof}
For the first statement, take $(\Sigma_1,\phi_1)$ and $(\Sigma_2,\phi_2)$ to be strictly right-veering open books and let $(\Sigma,\phi)$ the result of essential Murasugi sum of the two. We show that $\phi$ is strictly right-veering.
We consider $\phi_1$ and $\phi_2$ as diffeomorphisms of $\Sigma$ (by extending by the identity), and note that they are right-veering.
Let us now consider an essential arc $a$ in~$\Sigma$. We need to show that $\phi(a)=\phi_2(\phi_1(a))>a$.

Assume towards a contradiction that~$\phi(a)=\phi_2(\phi_1(a))=a$.
Since $\phi_1(a)\geq a$ and $\phi_2(\phi_1(a))\geq \phi_1(a)$ (by right-veeringness of~$\phi_1$ and~$\phi_2$, respectively), the inequalities $\phi(a)=\phi_2(\phi_1(a))\geq \phi_1(a)\geq a$, must both be equalities; in particular, we have~$\phi_1(a)=a$. However, together with strict right-veeringness of $\phi_1\colon \Sigma_1\to \Sigma_1$, this implies that $a$ can be isotoped (moving freely on the boundary) to an arc in~$\Sigma\setminus \Sigma_1$. In particular,~$\phi(a)=\phi_2(a)$. Moreover, since the Murasugi sum is essential, $a$ must be an essential arc in~$\Sigma_2$. However, now strict right-veeringness of $\phi_2\colon \Sigma_2\to \Sigma_2$, implies $(\phi(a)=)\phi_2(a)>a$, which contradicts~$\phi(a)=a$.

The more general statement for trees follows from the first part by induction, where the case of $T$ consisting of two vertices corresponds to the first part.
\end{proof}
\begin{remark}\label{rmk:connectiontoItosproof}
We note that \Cref{lemma:prime_trees} is enough to recover Cromwell's result `visual-primeness of positive braid links'. This relates to Ito's argument for Cromwell's result. Indeed, \Cref{lemma:prime_trees}, amounts to the observation that essential Murasugi sums of strictly right-veering open books are strictly right-veering. A special case of this observation---plumbing a positive Hopf band (in particular, a strictly right-veering open book) to an open book with positive monodromy (in particular, right-veering) that has no fixed essential arcs (hence is strictly right-veering), yields an open book that has no essential fixed arcs---was the input Ito used for his proof.
\end{remark}

For the general case of \Cref{thm:prime_trees}, we consider subtrees of $T$ for which the vertices are all either strictly left or strictly right-veering, do the corresponding Murasugi summing and consider a new tree $B$ in which the vertices consist of the result of these summings and edges are between open books of oppositely-handed veeringness. We then show that the open book associated with $B$ has no essential fixed arc by iterative application of \Cref{prop:nofixedarccriterion}. Here are the details. 

\begin{proof}[Proof of \Cref{thm:prime_trees}]
Let $T$ be a tree of open books, where each open book is strictly left-veering or strictly right-veering. As a first step, we consider the maximal subtrees $T_1,\dots,T_n$ of $T$ for which the vertices are either all strictly right-veering or they are all strictly left-veering.

Next we choose a growing of $T$ such that whenever we have grown a vertex that lies in a $T_i$ or we start with a vertex that lies in~$T_i$, for some~$1\leq i\leq k$, we grow the rest of the vertices that lie in $T_i$ before other vertices are grown. We prove that the open book associated to this tree with this growing has no essential fixed arcs.

This growing of $T$ restricts to growings of the~$T_i$. Let the corresponding open books be denoted by~$v_i=(\Sigma_i, \phi_i)$. Note that each $v_i$ is strictly right-veering or strictly left-veering as an iterative Murasugi sum of strictly right-veering or strictly left-veering open books; see \Cref{lemma:prime_trees}.
Let $B$ be the tree of open books with vertices the $v_i$ and edges and labeling induced by~$T$, the latter meaning the following. Whenever an edge $e=\{v,w\}$ of $T$ is such that $v$ is a vertex in some $T_i$ and $w$ is a vertex in a $T_j$ for some~$1\leq i<j \leq k$, then $\{v_i,v_j\}$ is an edge of $B$ and the labeling is the same as the labeling of $e$ with the embeddings of the polygon extended (by canonical inclusion) to the surfaces $\Sigma_i$ and~$\Sigma_j$. The tree $B$ is in fact bipartite with respect to grouping the vertices into strictly right-veering and strictly left-veering.

Next, we note that the growing of $T$ induces a growing of $B$ and that the open book associated with $T$ and its chosen growing can be canonically identified with the open book associated with $B$ and the corresponding growing. Hence, it remains to show that the open book associated with $B$ and its growing has no essential fixed arc.
Since each edge in $B$ is between open books where one is strictly right-veering and the other is strictly left-veering, each Murasugi sum in the iterative construction of the open book $(\Sigma,\phi)$ associated with $B$ and its growing satisfies \Cref{prop:nofixedarccriterion} (or its analog, where `left' and `right' are exchanged). Therefore, $(\Sigma,\phi)$ has no essential fixed arcs. \end{proof}

\subsection{Arborescent fibered links are prime}

As an application of \Cref{thm:prime_trees} and an illustration of examples of trees of open books, we discuss the primeness of arborescent fibered links.
We do so without recalling the construction of arborescent links in detail, as we see this as an illustration of the technique rather than a main result of the text. In particular, since primeness for most arborescent links can be derived from most of them being hyperbolic, as discussed in \cite{Futer_2008}, we do not claim originality but rather a different perspective.
For a general reference for the construction of arborescent links (certain links associated with a plane tree with integer labeled vertices) we refer to~\cite[Ch.~12]{bs}.
\emph{Fibered arborescent links} are the links $L_G$ in $S^3$ (well-defined up to ambient isotopy) associated to a plane tree $G$ with $\pm$--labeled vertices, i.e.~a choice of embedding of a finite tree with vertices labeled by elements in~$\{+,-\}$. The link $L_G$ can be described by placing a Hopf band in the plane at each vertex (with the sign of the Hopf band being dictated by the sign of the corresponding vertex) and plumbing them together whenever they share an edge; compare with~\cite{MR2165205} where this construction is explained, and consult~\cite[Ch.~12]{bs} for a detailed account. 
In the language of Murasugi sums, arborescent links are those arising from plumbing together oriented unknotted annuli (understood as open books for $S^3$). They are fibered if each of the involved annuli is fibered (since Murasugi sums have fibered boundary if and only if the pieces have fibered boundary~\cite[Theorem~4]{MurasugiSum}); hence, fibered arborescent links are those that arise if each annulus is a Hopf band.

It is not hard to see that every plane $\pm$--labeled tree $G$ in fact determines a tree of open books~$T$, where the vertices are Hopf band monodromies and the plumbing regions are neighborhoods of essential arcs in the annuli, such that the binding of the open book associated $T$ is, up to ambient isotopy,~$L_G$. The construction is detailed in the next paragraph.

The vertices of $T$ are chosen to be open books~$(A,\tau^\pm)$, where $A$ is an annulus and $\tau^\pm$ is a positive/negative Dehn twist, with sign chosen according to the label. The edges are chosen according to the edges of~$G$. The choice of labeling of the edges is where all the subtlety lies. The fact is that the plane tree $G$ induces a circular order on the edges at a vertex of $G$ (and hence of $T$). We choose the plumbing regions $P_1, \ldots, P_n$ in an annulus $A$ that is the surface of a vertex $v=(A,\tau^\pm)$ of $T$ and we take as the labels of the edges $e_1$ to $e_n$ that contain $v$ to be $P_1, \ldots, P_n$ in any order such that 
the circular order induced by the orientation of $A$ equals the circular order induced by~$G$. This process yields an open book well-defined up to equivalence, as the following argument shows. Any other choice of order of the $P_i$ that also respects the circular order described above, amounts to changing the plumbing regions by applying an orientation preserving diffeomorphism of~$A$, hence resulting in a tree of open books for which the corresponding open book is equivalent to that of~$T$.

\begin{remark}
    Note that without a choice of circular order at each vertex, e.g.~if $G$ was taken to be a $\pm$--labeled tree without preferred embedding into the plane, different choices of plane embeddings do actually lead to non-isotopic links in~$S^3$~\cite{gerber}.
    \end{remark}

    With this setup, the fact that fibered arborescent links are prime is immediate from \Cref{thm:prime_trees}.

\begin{proposition}
\label{prop:arblinksareprime}
All fibered arborescent links are prime.
\end{proposition}
    \begin{proof}[Proof of \Cref{prop:arblinksareprime}]
     Let $L$ be a fibered arborescent link, i.e.~up to isotopy, the link arises as the link $L_G$ associated with a plane $\pm$--labeled tree~$G$. Since the link $L_G$ is the binding of the open book defined (up to conjugacy) by the tree of open book $T$  built from $G$ as discussed above, it is prime by \Cref{thm:prime_trees}.
    \end{proof}

\section{Homogeneous braids} \label{sec:braids}

In general, a link diagram that arises as the closure of a braid diagram representing a non-prime link will not readily allow the observation that the link is not prime. While it is known that the necessary isotopy can be chosen to be well adapted to the braid axis---namely a composition of applications of the braid relations and conjugations and so-called exchange moves suffice~\cite[The composite braid theorem]{BirmanMenasco}---as far as the authors know there are no a priori bounds on how many applications of braid relations, conjugations and exchange moves are needed.

In contrast, \cref{thm:main} states that if a diagram of a link arises as the closure of a braid diagram that is homogeneous, then if the link represented by the diagram is not prime, this can be seen in the diagram by means of a decomposition circle. In this section, we first establish the notation and definitions relevant to~\cref{thm:main}, and then prove it.

\subsection{Preliminaries on decompositions and braids}

For a link~$L$, we call an embedded $2$--sphere $S\subset S^3$ a \emph{decomposition sphere}, if $S$ intersects $L$ transversely as follows: writing $S^3=B_1\cup B_2$ with $B_i\subset S^3$ compact balls with boundary~$S$, for both $i=1$ and~$i=2$, the closed arc component $a$ of $L\cap B_i$ is not boundary parallel in~$B_i\setminus (L\setminus a)$. In other words, for both $i=1$ and~$i=2$, $B_i\cap L$ is not isotopic to the result of split sum of a trivial $1$--tangle with a link. We call a link \emph{prime}, if it does not admit a decomposition sphere. We note that in this definition, the unknot is a prime link, and a split link is prime if and only if its split components are prime.

We call a circle $C \subset S^2$ a \emph{decomposition circle} of a link diagram~$D$, if it intersects $D$ transversely in two points that are not double points of~$D$, as follows: writing $S^2=D_1\cup D_2$ with $D_i\subset S^2$ compact disks with boundary~$C$, for both $i=1$ and~$i=2$, $D_i\cap D$ does not have a component that is a properly embedded closed interval.
We call a decomposition circle \emph{honest} if it gives rise to a decomposition sphere of the corresponding link. See \cref{fig:decompcircles} for an illustration of these notions.%
\begin{figure}[ht]
\centering
\includegraphics{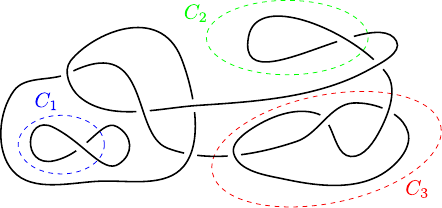}
\caption{A diagram of a split link with two components.
A circle $C_1$ (dotted, blue) that is not a decomposition circle; a dishonest decomposition circle~$C_2$ (dotted, green), and a honest decomposition circle~$C_3$ (dotted, red).}
\label{fig:decompcircles}
\end{figure}%

Artin's \emph{braid group on $n$ strands}~$B_n$ is given by the following presentation~\cite{Artin_25}.
\[
B_n\coloneqq\left\langle \sigma_1, \dots, \sigma_{n-1} \,\left|\,
\begin{array}{l}
\sigma_i\sigma_{i+1}\sigma_i = \sigma_{i+1}\sigma_i\sigma_{i+1} \text{ for } 1 \leq i \leq n-2,\\
\sigma_i\sigma_j = \sigma_j\sigma_i \text{ for } 1 \leq i \leq j - 2 \leq n - 3
\end{array}\right.
\right\rangle.
\]
For the sake of precision, 
we shall make a distinction between \emph{braid words}, which are elements of the free monoid on $\{\sigma_1^{\pm 1}, \dots, \sigma_{n-1}^{\pm 1}\}$, and the braids these words represent in~$B_n$.
By closing off strands, a word $W$ for a braid $\beta$ gives rise to a link diagram~$D(W)$.
The link represented by $D(W)$ is called the \emph{closure} of~$\beta$ or~$W$, denoted~$\cl(\beta)$ or~$\cl(W)$.

Let $W$ be a braid word on $n$ strands, and let~$1\leq i \leq j \leq n-1$.
Then we denote by $W_{ij}$ the braid word on $j-i+2$ strands obtained
from $W$ by deleting all $\sigma_k^{\pm 1}$ with $k < i$ or~$k > j$,
and shifting the index down by~$i-1$, i.e.~replacing each $\sigma_k^{\pm 1}$ with $i\leq k\leq j$ by~$\sigma_{k-i+1}^{\pm 1}$.

The word $W$ is called \emph{split} if for some~$i\in\{1,\ldots, n-1\}$, there are no $\sigma_i^{\pm 1}$ in~$W$.  If $W$ contains $\sigma_{n-1}^{\pm 1}$ precisely once,
then one may produce the $(n-1)$--stranded word $W_{1,n-2}$ by deleting~$\sigma_{n-1}^{\pm 1}$.
Similarly, if $W$ contains $\sigma_1^{\pm 1}$ precisely once,
one may produce $W_{2,n-1}$ by deleting $\sigma_1^{\pm 1}$ and shifting the index down by~$1$.
This is called \emph{(Markov) destabilization}, and its inverse \emph{(Markov) stabilization}.
Stabilization and destabilization leave the closure invariant.

\subsection{Decomposition circles of (homogeneous) braids}
To analyze decomposition circles of diagrams~$D(W)$,
we need to count `how often $W$ alternates between $\sigma_{i-1}^{\pm 1}$ and $\sigma_i^{\pm 1}$ generators'. More precisely, for a braid word $W$ on $n$ strands and~$2\leq i\leq n-1$, let the \emph{$i$--th seesaw number} $g_i(W)$ be defined as follows.
The word $W_{i-1,i}$
can be uniquely written as $V_1 \cdots V_{g_i(W)}$ for some~$g_i(W)\geq 0$, such that $V_{\ell}$ is a non-empty braid word consisting only of $\sigma_{i-1}^{\pm 1}$ if $\ell+\epsilon$ is even, and only of $\sigma_{i}^{\pm 1}$ if $\ell+\epsilon$ is odd, for some fixed~$\epsilon\in\{0,1\}$.

Let us recall the definition of homogeneity.
A braid word $W$ is called \emph{homogeneous}
if $W$ contains no $\sigma_i$ or no $\sigma_i^{-1}$ for each $i \in \{1,\ldots,n-1\}$.
A braid is \emph{homogeneous} if it can be written as a homogeneous braid word.
The following two properties of homogeneous braids quickly follow from their fiber structure~\cite{MR0520522}.
Homogeneous braids are \emph{visually trivial}, i.e.~a homogeneous word represents the unknot if and only if it contains each $\sigma_i^{\pm 1}$ exactly once~\cite{MR1002465}.
Moreover, homogeneous braids are \emph{visually split}, i.e.~a homogeneous word $W$ represents a split link if and only if $W$ is split~\cite{Cromwell_93}.

\begin{lemma}\label{lemma:decomposition_circles}
Let $W$ be a braid word on $n$ strands.
\begin{enumerate}[label=(\roman*)]
\item $D(W)$ admits a decomposition circle if and only if $g_i(W) \in \{2, 3\}$ for some~$i$.
\item %
Assume $W$ is homogeneous, non-split, and $D(W)$ admits no honest decomposition circle.
Let $W'$ obtained from $W$ by destabilizing as often as possible.
Then $g_i(W') \geq 4$ for all~$i$.
\end{enumerate}
\end{lemma}
\begin{proof}
\textbf{(i).} Let $C\subset S^2$ be a circle transverse to~$D(W)$.
After an isotopy pushing $C$ away from the crossings of~$D(W)$, all intersection points $C\cap D(W)$ are on one of the strands. It follows that $C$ intersects each strand an even number of times.
Thus if $C$ has two intersections with~$D(W)$, then $C$ intersects only one strand (let us say the $i$--th one), and does so twice.  One checks that this is possible if and only if~$i = 1$, $i = n$, or $2\leq i\leq n-1$ and~$g_i(W) \leq 3$. Moreover, $C$ is a decomposition circle if and only if $2 \leq i \leq n-1$ and~$g_i(W) \geq 2$.

\textbf{(ii).}
Denote by $m_-$ and $m_+$ the number of times $W$ has been destabilized at the lower and upper end, respectively, to obtain~$W'$.
Assume towards a contradiction that $g_i(W') \leq 3$ for some~$i$.
Since $W'$ is non-split,~$g_i(W') \geq 2$.
Hence, as proven above, there is a decomposition circle $C'$ of $D(W')$ that intersects the $i$--th strand of $W'$ twice. This gives rise to a decomposition circle $C$ of $D(W)$ that intersects the $(i + m_-)$--th strand of $W$ twice.
This circle decomposes the link $\cl(W)$ as a connected sum of $L_- \coloneqq \cl(W_{1,i+m_- -1})$ and~$L_+\coloneqq\cl(W_{i+m_-,n})$. Since $W'$ is non-split and cannot be destabilized, $W_{1,i+m_- -1}$ contains at least two copies of~$\sigma_{m_-+1}^{\pm 1}$. Similarly, $\cl(W_{i+m_-,n})$ contains at least two copies of~$\sigma_{n-m_+}^{\pm 1}$.
Using that homogeneous braids are visually trivial and visually split,
it follows that $L_+$ and $L_-$ are non-trivial non-split links,
and thus $C$ is an honest decomposition circle of~$W$.
\end{proof}

\subsection{Seifert surfaces of braids as Murasugi sums}

Let us now describe the Seifert surface $\Sigma(W)$ for~$\cl(\beta)$,
obtained by applying Seifert's algorithm to~$D(W)$.
The surface $\Sigma(W)$ consists of
disks~$D_1, \dots, D_n$, such that $D_i$ has radius~$n-i+1$, center~$(0,0,i)$, is parallel to the $xy$--plane, and its boundary is the $i$--th strand of~$W$;
as well as one twisted band for each crossing for each crossing of~$W$.
We call $D_1$ the \emph{bottom disk} and $D_n$ the \emph{top disk}.
The following lemma gives a precise version of the well-known decomposition of
$\Sigma(W)$ as Murasugi sum along the disks~$D_i$.
Examples can be seen in \cref{fig:braidanddiagr},
\cref{fig:murasugisforbetaprandbetadec}, and
\cref{fig:primsecMurasugisum} in \cref{sec:sketch}.
\begin{lemma} \label{lemma:braid_murasugi}
Let $W$ be a braid word on $n$ strands and let~$i \in \{2, \ldots, n-1\}$.
\begin{enumerate}[label=(\roman*)]
\item $\Sigma(W)$ is isotopic to a Murasugi sum of $\Sigma(W_{1,i-1})$ and~$\Sigma(W_{i,n-1})$.
\item The summing region $P$ is embedded into the top disk of $\Sigma(W_{1,i-1})$ and into the bottom disk of~$\Sigma(W_{i,n-1})$.
\item $P$ has $g_i(W)$ sides if $g_i(W)$ is even, and $g_i(W)-1$ sides if $g_i(W)$ is odd.
\item If~$g_i(W) \geq 4$, then the Murasugi sum is essential.
\end{enumerate}
\end{lemma}
\begin{proof}
Cutting $\Sigma(W)$ in half along the plane containing $D_i$
decomposes $\Sigma(W)$ as a union along $D_i$ of two Seifert surfaces:
a lower one isotopic to $\Sigma(W_{1,i-1})$ and an upper one isotopic to~$\Sigma(W_{i,n-1})$.
This proves (i) and~(ii), using the ambient definition of Murasugi sum as discussed in \Cref{rmk:visualMS}. %

Let us now explicitly define~$P$.
Recall from the definition of $g_i(W)\eqqcolon g$ that $W_{i-1,1} = V_1\cdots V_{g}$,
where the words $V_{\ell}$ consist alternatingly only of~$\sigma_{i-1}^{\pm 1}$,
or only of~$\sigma_{i}^{\pm 1}$.
Now, let us pick points $p_1, \ldots, p_{g} \in \partial\Sigma(W)$
that lie in this order on the circle~$\partial D_i$,
such that $p_j$ is between the spots where the twisted band corresponding to the last letter in $V_j$ and the twisted band corresponding to the first letter in $V_{j+1}$ are affixed to $D_i$
(where we see indexes modulo $g$ so that~$V_{g+1} = V_1$).
Let the polygon $P$ now be given a subset of the disk~$D_i$, with vertices $p_1, \ldots, p_{g-1}$
(and $p_g$ if $g$ is even), and edges given as chords of~$D_i$.
One checks that after an isotopy, the sides of $P$ indeed belong alternatingly to
the boundary of the lower and the upper Murasugi summand of~$\Sigma(W)$,
thus proving~(iii).

Finally, to show~(iv), let us assume~$g \geq 4$. Let $1\leq j\leq g$ and $j < g$ if $g$ is odd.
Let us treat indexes modulo $g$ if $g$ is even, and modulo $g-1$ if $g$ is odd.
Pick two twisted bands, such that both connect $D_i$ and~$D_{i-1}$, or both connect $D_i$ and~$D_{i+1}$: one of them affixed to $D_i$ between $p_j$ and~$p_{j+1}$, the other between $p_{j+2}$ and~$p_{j+3}$.
Pick a simple closed curve $\gamma_j$ in $\Sigma(W)$ that consists of the core curves of those two twisted bands, and a chord in each of $D_i$ and~$D_{i\pm 1}$.
Observe that $\gamma_j$ is either contained in the lower or in the upper Murasugi summand of~$\Sigma(W)$, depending on the sign in~$D_{i\pm 1}$.
Moreover, $\gamma_j$ intersects the side of $P$ between $p_j$ and $p_{j+1}$ exactly once.
That is sufficient to prove that this side is essential.
Overall, the sides of $P$ are alternatingly essential in the two Murasugi summands of~$\Sigma(W)$; hence, the Murasugi sum is essential by definition.
\end{proof}

\subsection{Proof of \texorpdfstring{\cref{thm:main}}{Theorem \ref{thm:main}}}
We are now ready to prove visual primeness of homogeneous braids.
\begin{proof}[Proof of \cref{thm:main}]
Let us prove the contrapositive:
given a homogeneous braid word $W$ with closure a link~$L$,
such that $D(W)$ does not admit an honest decomposition circle,
we shall show that $L$ is prime.

Because a link is prime if all of its split components are,
the case of split $W$ follows from the case of non-split~$W$,
e.g.~by induction over the number of strands.
So let us assume henceforth that $W$ is non-split.

Then $W$ satisfies the hypothesis of \cref{lemma:decomposition_circles}~(ii).
So, for $W'$ obtained from $W$ by destabilizing as often as possible,
we have $g_i(W') \geq 4$ for all~$i$. Note that $\Sigma(W)$ and $\Sigma(W')$ are isotopic.
Hence it suffices to prove that $\Sigma(W')$ is prime, i.e. that its associated open book for $S^3$ is prime.
Also note that $W'$ is itself homogeneous and non-split. Let $n$ be the number of strands of~$W'$.

An iterative application of \cref{lemma:braid_murasugi}
decomposes $\Sigma(W')$ as an $(n-1)$--fold Murasugi sum
with summands~$\Sigma(W'_{i,i})$.
Because $W'$ is homogeneous, every $W'_{i,i}$ consists either of $k$ copies of~$\sigma_1$, or of $k$ copies of~$\sigma_1^{-1}$, with~$k\geq 2$. So $\Sigma(W'_{i,i})$ is the fiber surface of the $T(2,\pm k)$--torus link.

This Murasugi sum decomposition of $\Sigma(W')$ is encoded by a tree of open books~$T$
as defined in Section~\ref{sec:murasugi_sum}. Namely, the vertices of $T$ are the open books of the~$\Sigma(W'_{i,i})$;
and there is an edge precisely between $\Sigma(W'_{i-1,i-1})$ and $\Sigma(W'_{i,i})$ for~$2\leq i \leq n - 1$.
Note that the tree $T$ simply has the shape of a line.
The edge between $\Sigma(W'_{i-1,i-1})$ and $\Sigma(W'_{i,i})$
is labeled by the summing region $P_i$ given by \cref{lemma:braid_murasugi}.
Let us check that $T$ satisfies the hypothesis of a tree of open books.
Firstly, the summing regions are disjoint, because $P_i$ is embedded in the $i$--th disk $D_i$ of~$\Sigma(W')$. Secondly, since $g_i(W') \geq 4$ for all~$i$, it follows from \cref{lemma:braid_murasugi}~(iv)
that $P_i$ is essential. So $T$ is indeed a tree of open books.
The theorem now follows from \cref{thm:prime_trees}.
\end{proof}

\section{Trefoil plumbings do not preserve primeness} \label{sec:trefoilplumbing}

We will now prove \Cref{prop:nonprimnessoftrefoilplumbing}, which involves the construction of an open book $(\Sigma_\rho,\rho)$ and an essential arc $a$ such that trefoil plumbing along $a$  yields an open book that has a fixed essential separating arc~$\gamma$.
Further modifications will allow us to see that we can arrange for the binding to be connected and for the original open book to have no fixed essential arcs.

\begin{proof}[Proof of \Cref{prop:nonprimnessoftrefoilplumbing}]
Consider the genus 2 surface with two boundary components. We denote it by $\Sigma_\rho$ and pick an explicit model: we view $\Sigma_\rho$ as the result of surgering one-holed tori onto an annulus at two spots given by two marked points; see left-hand side of \Cref{fig:SigmaP}. \begin{figure}[b]
\centering
    \includegraphics[width=0.9\linewidth,origin=c]{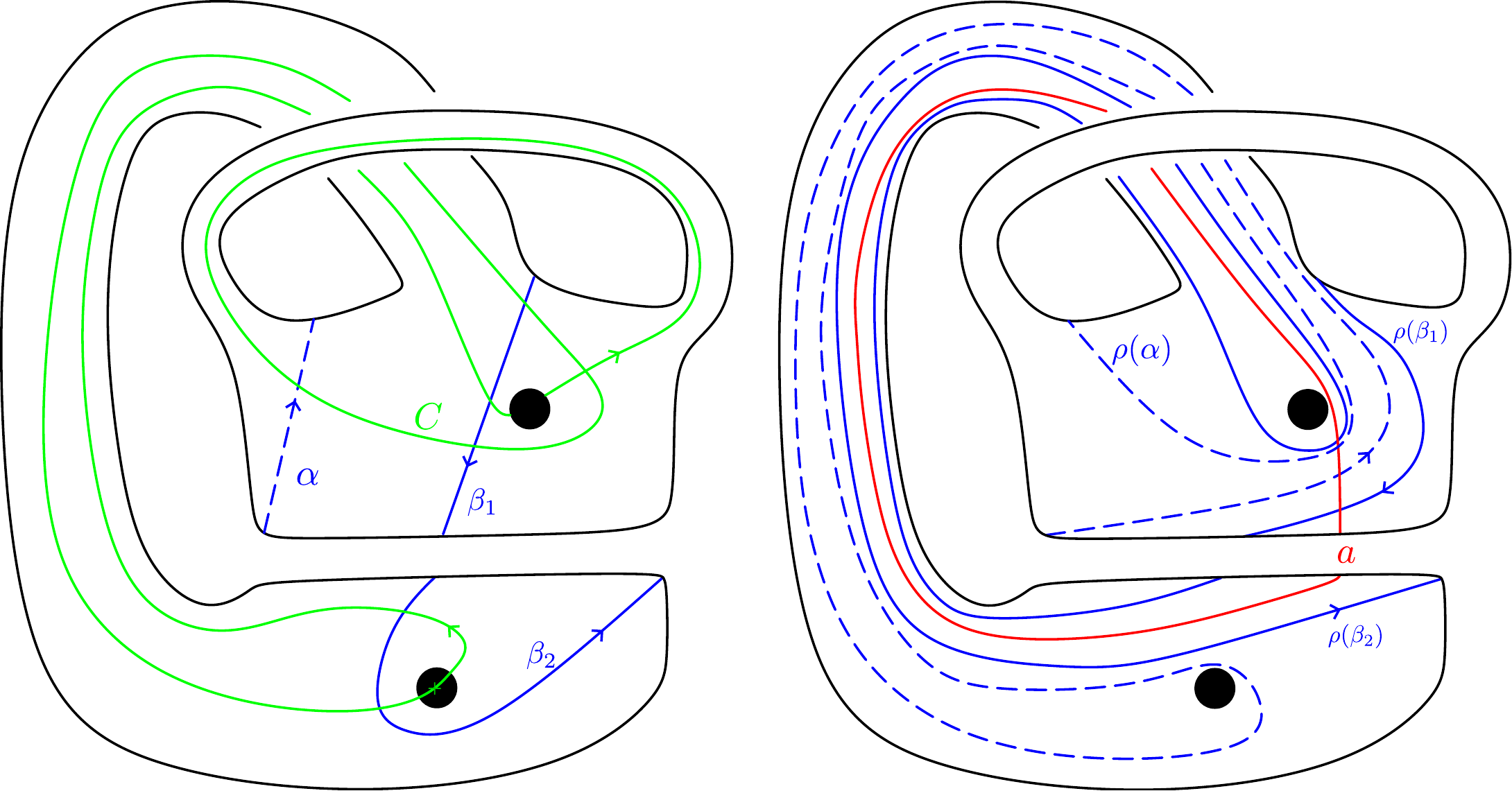}
    \caption{Left: The surface $\Sigma_{\rho}$, depicted as an annulus with two discs (black) which indicate where two once punctured tori are surgered in, and three essential arcs $\alpha$ (dotted, blue), $\beta_1$ (blue), and $\beta_2$ (blue). Also an immersed circle $C$ (green) in the annulus that runs through both discs. 
    Right: The images of $\alpha$, $\beta_1$,  and $\beta_2$ under~$\rho$, and the stabilization arc $a$ (red).}
\label{fig:SigmaP}
  \end{figure}
  We also pick three essential arcs $\alpha$, $\beta_1$, and $\beta_2$ in~$\Sigma_{\rho}$. We will use those
arcs together with the cores of the handle of the plumbings to obtain the fixed arc~$\gamma$.

We define the monodromy $\rho$ in two steps. Firstly, we define it as a diffeomorphism of the annulus with two marked points. For this we pick $\rho(\alpha)$, $\rho(\beta_1)$, and $\rho(\beta_2)$ as depicted on the right-hand side of \Cref{fig:SigmaP}. There exists a diffeomorphism mapping $\alpha$, $\beta_1$, and $\beta_2$ to $\rho(\alpha)$, $\rho(\beta_1)$, and $\rho(\beta_2)$ that fixes the boundary pointwise and the two marked points as a set, and this diffeomorphism, which we denote by~$\rho$, is unique up to isotopy relative boundary. To see this, we note that $\alpha$, $\beta_1$, $\beta_2$ cut the annulus into a disc (in fact a $6$--gon) and discs with one marked point (one bigon and one punctured $4$--gon each with one marked point), and the same is true for $\rho(\alpha)$, $\rho(\beta_1)$, $\rho(\beta_2)$.\footnote{As an alternative description of~$\rho$, the reader may consider the immersed oriented circle $C$ with a marked point in the annulus depicted on the left-hand side of \Cref{fig:SigmaP}, and take $\rho$ to be the half-push map that exchanges the two marked points, starting with the one marked on $C$ and pushing in the direction of the orientation. We will not use the latter description of~$\rho$.}
Secondly, we choose an extension of $\rho$ to a self--diffeomorphism of~$\Sigma_\rho$, which, by abuse of notation, we denote as $\rho\colon \Sigma_\rho \to \Sigma_\rho$. By construction, $\rho$ interchanges the two one-holed tori surgered to the annulus at the marked points.

Next, we choose an essential arc $a$ as depicted in the right-hand side of \Cref{fig:SigmaP}, along which we perform a positive stabilization yielding an open book~$(\Sigma_1,\phi_1)$.
By viewing $\Sigma_\rho$ as a subsurface of~$\Sigma_1$, $\Sigma_1$ naturally features a copy of the arc $\alpha$ but furthermore features an arc $\beta$ given as the union of $\beta_1$ and $\beta_2$ with a core of the newly attached handle $\Sigma_1\setminus \mathrm{int}(\Sigma_\rho)$; see the left of \Cref{fig:Sigma1}.
Extending $\rho$ to $\Sigma_1$ by the identity, the resulting monodromy $\phi_1$ can be understood as the composition of $\rho$ with a positive Dehn twist $\tau_{a'}$ along the simple closed curve given as the union of $a$ with a core of the new handle. The result of applying $\phi_1$ to $\alpha$ and $\beta$ is depicted on the right in \Cref{fig:Sigma1}.
\begin{figure}[ht]
\centering
    \includegraphics[width=0.9\linewidth,origin=c]{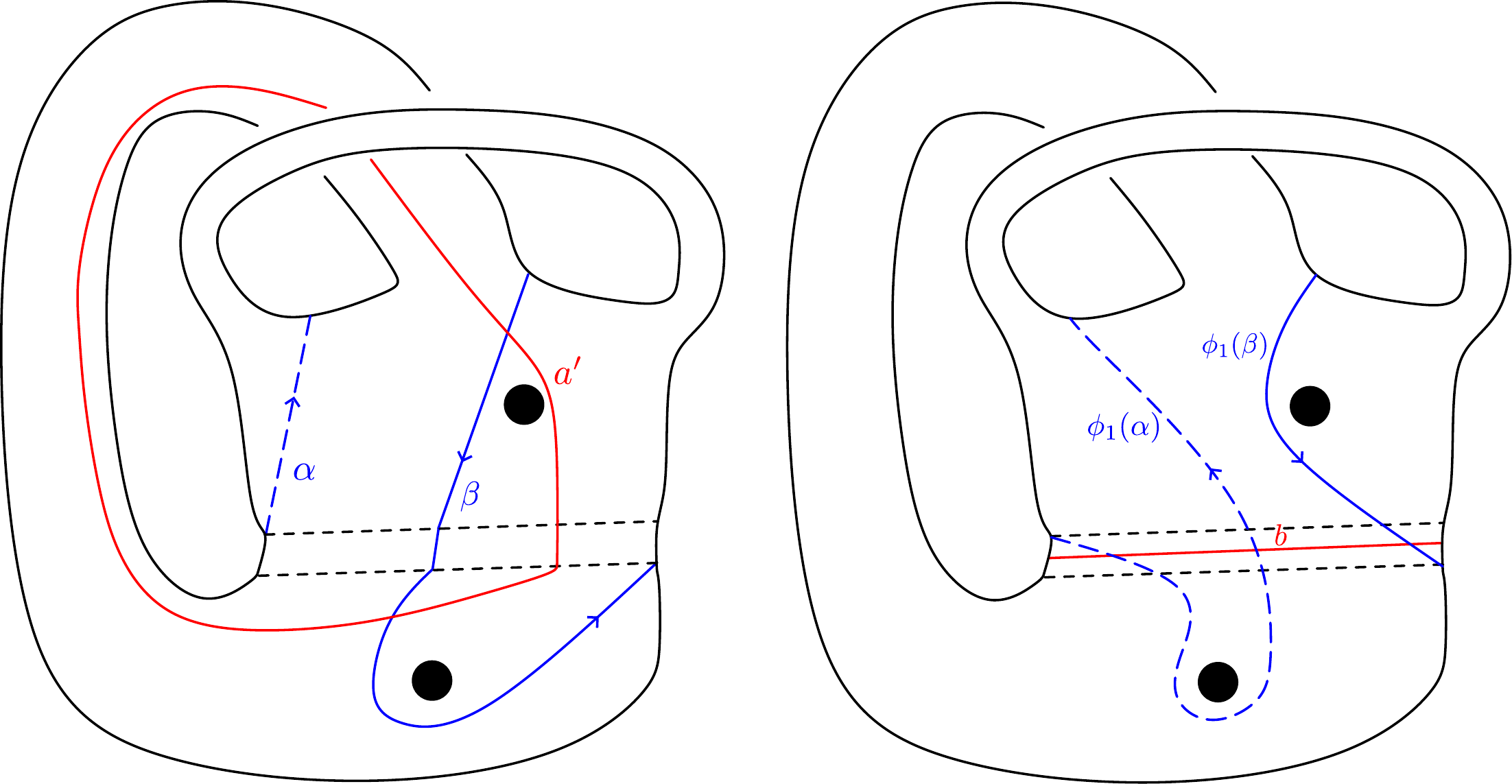}
    \caption{Left: the arcs $\alpha$ and $\beta$ on~$\Sigma_1$, together with the curve $a'$ (obtained by joining $a$ with the core of the handle). The new handle corresponding to the stabilization is bounded by the dashed lines. Right: the images $\phi_1(\alpha) = \tau_{a'}(\rho(\alpha))$ and $\phi_1(\beta) = \tau_{a'}(\rho(\beta))$; and the cocore of the handle $b$ along which we perform the second stabilization.}
\label{fig:Sigma1}
  \end{figure}

We also consider the arc $b$ in $\Sigma_1$ as given in \Cref{fig:Sigma1}, and let $(\Sigma,\phi)$ be the open book resulting from  $(\Sigma_1,\phi_1)$ by positive stabilization along~$b$. Since $b$ is a cocore of the handle $\Sigma_1\setminus \mathrm{int}(\Sigma_\rho)$, the open book $(\Sigma,\phi)$ is equivalent to the trefoil plumbing of $(\Sigma_\rho,\rho)$ along~$a$. We next describe a fixed essential separating arc $\gamma$ of~$(\Sigma,\phi)$.

Note that $\alpha$ and $\beta$ are chosen such that they each have one endpoint in a corner of the plumbing region. Thinking of $\Sigma_1$ as a subset of~$\Sigma$, by a small isotopy we can assume that the arcs have that endpoint in a part of $\partial \Sigma_1$ that lies in the interior of~$\Sigma$. We define the arc $\gamma$ in $\Sigma$ to be the union of $\alpha$, $\beta$, and a core $h$ of the newly attached handle $\Sigma\setminus \mathrm{int}(\Sigma_1)$, see the left hand side of \Cref{fig:second_plumbing}.
One readily checks that $\gamma$ is an essential separating arc, and it is not hard to see that $\gamma$ is fixed. Indeed, note for example~$\phi_1(\gamma)$, which equals the union of $\phi_1(\alpha)\cup \phi_1(\beta)\cup h$, is isotopic to the result of applying~$\phi_2^{-1}$. Here $\phi_1\colon \Sigma\to\Sigma$ is the extension by the identity of the monodromy of   $(\Sigma_1,\phi_1)$ and $\phi_2\colon \Sigma\to\Sigma$ is the extension by the identity of the positive Dehn twist that is the monodromy of the Hopf band that was plumbed in the stabilization along~$b$.
A representative of $\phi_1(\gamma) = \phi_2^{-1}(\gamma)$, chosen so that it differs from $\gamma$ only in the plumbing region, can be seen on the right-hand side of \Cref{fig:second_plumbing}.

\begin{figure}[ht]
\centering
    \includegraphics[width=0.9\linewidth,origin=c]{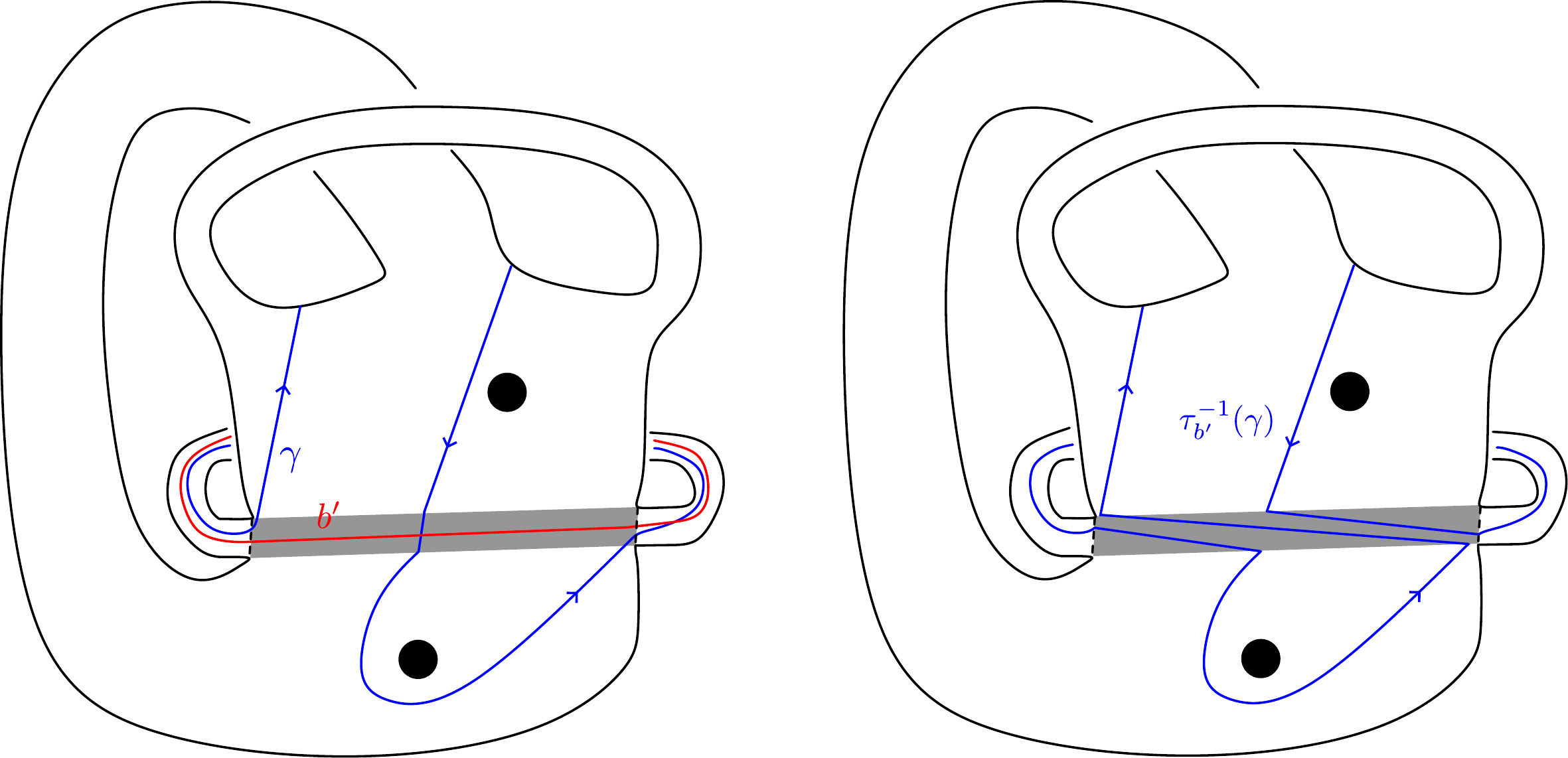}
    \caption{Left: the surface~$\Sigma$, and on it the arc $\gamma$ (the union of~$\alpha$, $\beta$ and the curve $b'$ obtained by joining $b$ and the core of the new handle), and the plumbing region (grey). Right: the arc $\tau_{b'}^{-1}(\gamma) = \phi_2^{-1}(\gamma) = \phi_1(\gamma)$. Note that, following \Cref{lemma:lemma1}, we have chosen a representative of its isotopy class such that it differs from $\gamma$ only in the plumbing region.}
\label{fig:second_plumbing}
  \end{figure}

It remains to discuss that the open book $(\Sigma_\rho, \rho)$ has no fixed essential arcs, and to establish that we can arrange for the open book to have connected binding.

For the latter, assuming $(\Sigma_\rho, \rho)$ has no essential fixed arcs, we can consider the open book $(\Sigma_0,\phi_0)$ resulting from iteratively plumbing three Hopf bands to $(\Sigma_\rho, \rho)$ along an arc $c$ connecting two boundary components such that $c$ is disjoint from $a$, $\rho(\alpha)$, $\rho(\beta_1)$, $\rho(\beta_2)$.
More precisely, consider the tree $T$ of open books with vertices $v_1=(\Sigma_\rho, \rho)$, $v_2$, $v_3$, and~$v_4$, where $v_2$ and $v_4$ are open books of positive Hopf bands, while $v_3$ is the open book of a negative Hopf band, and take the edge set to be
$\{\{v_1,v_2\},\{v_2,v_3\},\{v_3,v_4\}\}$. 
Further take the labelings $P_1$, $P_2$, and $P_3$ such that $P_1$ is the plumbing region in $\Sigma_\rho$ given as a $4$--gon that is a neighborhood of $c$ (all other plumbing regions are uniquely determined up to isotopy, since in annuli there is only one essential $4$--gon plumbing region up to isotopy).
We claim that the open book $(\Sigma_0,\phi_0)$ (well-defined up to equivalence) associated with $T$ has no fixed essential arc.
To see this we note that, if $T$ had one less vertex, namely~$v_4$, and one less edge, namely $\{v_3,v_4\}$, then the corresponding open book $(\Sigma_0',\phi_0')$ is an essential figure-eight plumbing, which has no fixed essential arc by \Cref{prop:primnessoffigure8plumbing}.  Given that $T$ has the extra vertex and edge, the associated open book is in fact the result of positive stabilization of $(\Sigma_0',\phi_0')$. Now, the plumbing arc for the stabilization veers to the left in $(\Sigma_0',\phi_0')$, $(\Sigma_0',\phi_0')$ has no fixed essential arcs, and the stabilization is positive (meaning we are summing a strictly right-veering open book to $(\Sigma_0',\phi_0')$). Thus, \Cref{prop:nofixedarccriterion} implies that the open book $(\Sigma_0,\phi_0)$ has no fixed essential arcs. By construction $\Sigma_0$ has connected boundary; hence $(\Sigma_0,\phi_0)$ is an open book with connected binding as desired.

Finally, we slightly modify $(\Sigma_\rho, \rho)$ to assure it does not contain fixed essential arcs. First, we argue that we may assume that all essential fixed arcs are disjoint from the interior of $\alpha$, $\rho(\alpha)$ $\rho(\beta_1)$, and~$\rho(\beta_2)$.
To see this, we simply modify $(\Sigma_\rho,\rho)$ by performing figure-eight plumbings along (parallel copies of) the essential arcs $\rho(\alpha)$, $\rho(\beta_1)$, and~$\rho(\beta_2)$.
The result is a new open book $(\Sigma_0,\phi_0)$ with the property that all essential fixed arcs can be isotoped to miss the plumbing arcs by \Cref{rem:fixarcsinfig8}. In particular, if $(\Sigma_0,\phi_0)$ has an essential fixed arc, then $(\Sigma{\rho}P,\rho)$ has an essential fixed arc disjoint from $\rho(\beta_1)$, $\rho(\beta_2)$, and~$\rho(\alpha)$. 
Furthermore, note that $\phi_0$ and $\rho$ map $\alpha$, $\beta_1$, and $\beta_2$ to the same curves since we chose figure-eight plumbing arcs that are disjoint from $\rho(\alpha)$, $\rho(\beta_1)$, and~$\rho(\beta_2)$.
Hence, the construction of the fixed arc $\gamma$ goes through for $(\Sigma_0,\phi_0)$ in the exact same way as for~$(\Sigma_\rho,\rho)$. 

With this we have reduced to the problem of showing that  $(\Sigma_\rho,\rho)$ has no fixed essential arc that is disjoint from $\rho(\beta_1)$, $\rho(\beta_2)$ and $\rho(\alpha)$. We observed earlier that $\Sigma_\rho\setminus (\rho(\beta_1)\cup \rho(\beta_2)\cup \rho(\alpha))$ has three connected components. We discuss that none of these contain fixed essential arcs.
One component is a $6$--gon bounded by segments in $\partial \Sigma_\rho$ and all three arcs $\rho(\beta_1)$, $ \rho(\beta_2)$, $\rho(\alpha)$. Every arc in this $6$--gon that is not boundary parallel in $\Sigma_\rho$ is parallel to one of $\rho(\beta_1)$, $\rho(\beta_2)$, $\rho(\alpha)$, none of which are fixed.
A second component is a bigon with a marked point, where the marked point is replaced by a one-holed torus, that is bounded by a segment in $\partial \Sigma_\rho$ and~$\rho(\beta_1)$. A fixed essential arc $\gamma$ must necessarily enter the one-holed torus, as otherwise it is parallel to the boundary (and hence not essential) or to~$\rho(\beta_1)$, and the latter is not fixed.
Here, entering the one-holed torus is formalized as having non-trivial intersection with the simple closed curve encircling the marked point. However, such a $\gamma$ is disjoint from the curve that encircles the other marked point---respectively one-holed torus---, but $\rho(\gamma)$ intersects said curve since $\rho$ exchanges the two curves that encircle the marked points. In particular, $\gamma$ and $\rho(\gamma)$ are not isotopic.
The third component is a $4$--gon with a marked point, where the marked point is replaced by a one-holed torus, that is bounded by segments in $\partial \Sigma_\rho$, $\rho(\beta_2)$ and~$\rho(\alpha)$. A similar argument as above shows that any fixed essential $\gamma$ must enter the one-holed torus, but does, of course not enter the other one-holed torus, hence $\gamma$ cannot be fixed as $\rho$ exchanges the two one-holed tori. This concludes the proof that we may choose the open book so that it does not contain fixed essential arcs.
\end{proof}

\bibliographystyle{myamsalpha} 
\bibliography{References}
\end{document}